\documentclass[smallextended]{svjour3}

\smartqed

\usepackage{aicescover}

\aicescovertitle{Legendre-Gauss-Lobatto grids\\ and associated\\ nested dyadic grids}
\aicescoverack{Financial support by the DFG project ``Optimal preconditioners of spectral Discontinuous Galerkin methods for elliptic boundary value problems'' (DA 117/23-1), the Excellence Initiative of the German federal and state governments (RWTH Aachen Seed Funds, Distinguished Professorship projects, Graduate School AICES), and NSF grant DMS 1222390 is gratefully acknowledged.}

\usepackage[utf8]{inputenc}
\usepackage{amsmath}
\usepackage{amssymb}
\usepackage{graphicx,subfig}
\usepackage{enumitem}
\usepackage{algorithm,algpseudocode}

\usepackage{hyperref}
\usepackage{bookmark}
\hypersetup{
  pdftitle    = {Legendre-Gauss-Lobatto grids and associated nested dyadic grids},
  pdfauthor   = {Kolja Brix, Claudio Canuto, and Wolfgang Dahmen},
  pdfkeywords = {Legendre-Gauss-Lobatto grid, dyadic grid, graded grid, nested grids},
  pdfborder={0 0 0.5}
}

\def\makeheadbox{}

\author{\mbox{Kolja Brix $\cdot$ Claudio Canuto $\cdot$ Wolfgang Dahmen}}
\authorrunning{K. Brix, C. Canuto, and W. Dahmen}

\title{Legendre-Gauss-Lobatto grids and\\ associated nested dyadic grids}

\institute{
K. Brix \and W. Dahmen \at
Institut f\"ur Geometrie und Praktische Mathematik, RWTH Aachen, Templergraben 55, 52056 Aachen, Germany\\
\email{\{brix,dahmen\}@igpm.rwth-aachen.de}
\and
C. Canuto \at
Dipartimento di Scienze Matematiche, Politecnico di Torino, Corso Duca degli Abruzzi 24, 10129 Torino, Italy\\
\email{claudio.canuto@polito.it}
}

\date{}

\newcommand{\fsabs}[1]{\ensuremath{|#1|}} 
\newcommand{\abs}[1]{\ensuremath{\left|#1\right|}} 
\newcommand{\fsfloor}[1]{\ensuremath{\lfloor #1 \rfloor}}
\newcommand{\floor}[1]{\ensuremath{\left\lfloor #1 \right\rfloor}}
\newcommand{\ceil}[1]{\ensuremath{\left\lceil #1 \right\rceil}}

\newcommand{\R}{\ensuremath{\mathbb{R}}} 
\newcommand{\N}{\ensuremath{\mathbb{N}}} 

\newcommand{\cG}{\ensuremath{{\cal G}}}
\newcommand{\cD}{\ensuremath{{\cal D}}}
\newcommand{\cP}{\ensuremath{{\cal P}}}

\newcommand{\oDelta}{\overline{\Delta}}
\newcommand{\uDelta}{\underline{\Delta}}
\newcommand{\oeta}{\overline{\eta}}
\newcommand{\ueta}{\underline{\eta}}

\newcommand{\cGLGL}{\ensuremath{\cG^\text{LGL}}}
\newcommand{\cGCGL}{\ensuremath{\cG^\text{CGL}}}

\newcommand{\SN}{${\bf Str}_N\, $}
\newcommand{\SbarN}{${\bf Str}_{\bar N}\, $}
\newcommand{\MQ}{${\bf MQ}_{\bar N}\, $}

\newcommand{\algo}[1]{\textsf{#1}}

\DeclareMathOperator{\argmax}{argmax}
\DeclareMathOperator{\argmin}{argmin}

\begin{document}

\aicescoverpage

\setcounter{page}{1}

\maketitle

\begin{abstract}
Legendre-Gauss-Lobatto (LGL) grids play a pivotal role in nodal spectral methods for the numerical solution of partial differential equations. They not only provide efficient high-order quadrature rules, but give also rise to norm equivalences that could eventually lead to efficient preconditioning techniques in high-order methods. Unfortunately, a serious obstruction to fully exploiting the potential of such concepts is the fact that LGL grids of different degree are not nested. This affects, on the one hand, the choice and analysis of suitable \emph{auxiliary spaces}, when applying the \emph{auxiliary space method} as a principal preconditioning paradigm, and, on the other hand, the efficient solution of the auxiliary problems. As a central remedy, we consider certain nested hierarchies of dyadic grids of locally comparable mesh size, that are in a certain sense properly associated with the LGL grids. Their actual suitability requires a subtle analysis of such grids which, in turn, relies on a number of refined properties of LGL grids. The central objective of this paper is to derive just these properties. This requires first revisiting properties of close relatives to LGL grids which are subsequently used to develop a refined analysis of LGL grids. These results allow us then to derive the relevant properties of the associated dyadic grids.



\end{abstract}


\section{Introduction}\label{sect:intro}

Studying the distribution of zeros of orthogonal polynomials is a classical theme that has been addressed in an enormous number of research papers. The sole reason for revisiting this topic here is the crucial role of \emph{Legendre-Gauss-Lobatto (LGL) grids}, formed by the zeros of corresponding orthogonal polynomials, for the development of efficient preconditioners for high order finite element and even spectral discretizations of PDEs, see e.g. \cite{Ca94,CHQZ2006,BCCD1,BCCD2}. In connection with elliptic PDEs, which often possess very smooth solutions and thus render high order methods at least potentially extremely efficient, a key constituent is the fact that interpolation at LGL grids give rise to fully robust (with respect to the polynomial degrees) isomorphisms between high order polynomial spaces and low order finite element spaces on LGL grids. However, unfortunately, one quickly faces some serious obstructions to fully exploiting this remarkable potential of LGL grids when simultaneously trying to exploit the flexibility of \emph{Discontinuous Galerkin} (DG) schemes, namely locally refined grids and locally varying polynomial degrees. Indeed, as explained in \cite{BCCD1}, the essential source of the problems encountered then is the fact that LGL grids are \emph{not nested}. This affects the choice and analysis of suitable \emph{auxiliary spaces}, when using the auxiliary space method, see e.g. \cite{Oswald,Xu} as preconditioning strategy, as well as the efficient solution of the corresponding auxiliary problems. As a crucial remedy, certain hierarchies of nested \emph{dyadic} grids have been introduced in \cite{BCCD1} which are \emph{associated} in a certain sense with LGL grids. The term ``associated'' encapsulates a number of properties of such dyadic grids, some of which have been used and claimed in \cite{BCCD1} but will be proved here which is the central objective of this paper.

The layout of the paper is as follows. After collecting, for the convenience of the reader, in Section~\ref{sect:ultraspherical} some classical facts and tools used in the sequel, we formulate in Section~\ref{sect:main} the main results of the paper. The first one, Theorem~\ref{thm:A}, is concerned with the quasi-uniformity of LGL grids as well as with a certain notion of equivalence between LGL grids of \emph{different} order. This is important for dealing with DG-discretizations involving locally varying polynomial degrees, see \cite{BCCD1,BCCD2}. The second one, Theorem~\ref{thm:B}, concerns certain hierarchies of nested dyadic grids that are associated in a very strong sense with LGL grids. Both theorems play a crucial role for the design and analysis of preconditioners for DG systems. Section~\ref{sect:ProofA} is devoted to the proof of Theorem~\ref{thm:A}. This requires deriving a number of \emph{refined properties} of LGL grids which to our knowledge cannot be found in the literature. In particular, we need to revisit in Section~\ref{sect:CGL} some close relatives namely \emph{Chebyshev-Gauss-Lobatto (CGL) nodes} since they have explicit formulae that help deriving sharp estimates. The central subject of Section~\ref{sect:dyadic} is the generation of dyadic grids associated in a certain way with a given other grid as well as the analysis of the properties of these dyadic grids. In particular, the results obtained in this section lead to a specific hierarchy for which the properties claimed in Theorem~\ref{thm:B} are verified.

Throughout the paper we shall employ the following notational convention. By $a \lesssim b$ we mean that the quantity $a$ can be bounded by a constant multiple of $b$ uniformly in the parameters $a$ and $b$ may depend on. Likewise $a \gtrsim b$ is equivalent to $b \lesssim a$ and $a\simeq b$ means $a\lesssim b$ and $b\lesssim a$.


\section{Preliminaries and Classical Tools}\label{sect:ultraspherical}

A central notion in this work concerns \emph{grids} $\cG$ induced by zeros of certain orthogonal polynomials, especially, the first derivatives of Legendre polynomials. It will be important though that those are special cases of \emph{Ultraspherical or Gegenbauer polynomials} whose definition is recalled for the convenience of the reader.

\begin{definition}[Ultraspherical or Gegenbauer polynomials, {\cite[Section~4.7]{Szegoe1978}}]
Let the parameters $\lambda > -\frac{1}{2}$ and $N\in\N$ be fixed. The ultraspherical or Gegenbauer polynomial $P^{(\lambda)}_N$ of degree $N$ is defined as orthogonal polynomial on the interval $[-1,1]$ with respect to the weighting function $w^{(\lambda)}(x):=(1-x^ 2)^{\lambda-\frac{1}{2}}$, i.e.
\begin{align*}
\int_{-1}^{1} P^{(\lambda)}_{N}(x) P^{(\lambda)}_{N'}(x) \, w^{(\lambda)}(x) \, \mathrm{d}x = c^{(\lambda)}_{N} \delta_{N,N'} \qquad \text{for all} \quad N,N' \in \N,
\end{align*}
with
\begin{align*}
c^{(\lambda)}_{N}=\frac{2^{1-2\lambda} \pi}{\Gamma(\lambda)^2} \frac{\Gamma(N+2\lambda)}{(N+\lambda)\Gamma(N+1)}
\end{align*}
and normalization
\begin{align*}
P^{(\lambda)}_N(1) = \binom{N+2\lambda-1}{N}.
\end{align*}
\end{definition}

The following useful properties of ultraspherical polynomials can be found in \cite[Chapter 4.7]{Szegoe1978}.
For any $N\in\N$ the ultraspherical polynomials fulfill the symmetry property
\begin{equation}
P^{(\lambda)}_N(-x) = (-1)^N \, P^{(\lambda)}_N(x) \qquad \text{for all} \quad x \in \R,
\label{eq:symmetry}
\end{equation}
and the differentiation rule
\begin{equation}
\frac{d}{dx} P^{(\lambda)}_N(x) = 2 \lambda \, P^{(\lambda+1)}_{N-1}(x) \qquad \text{for all} \quad x \in \R \label{eq:diffrule}
\end{equation}
holds.

The understanding of these polynomials hinges to a great extent on a fact that will be used several times, namely that the ultraspherical polynomial $P^{(\lambda)}_N$ is a solution of the linear homogeneous ordinary differential equation (ODE) of second order
\begin{equation}
(1-x^2) \, y''(x) - (2\lambda+1) x \, y'(x) + N(N+2\lambda) \, y(x) =0,
\label{eq:JacobiODE}
\end{equation}
see, e.g., \cite[(4.2.1)]{Szegoe1978}.
By an elementary transformation, based on the product ansatz for the solution $y(x):=s(x)u(x)$, the first order term can be eliminated, see \cite[Section~1.8]{Szegoe1978}, so that the transformed ODE
\begin{equation}
\frac{d^2u}{dx^2} + \phi^{(\lambda)}_{N}(x) u=0\;, \quad \text{where \ }
\phi^{(\lambda)}_{N}(x) = \frac{1-(\lambda-\frac{1}{2})^2}{(1-x^{2})^2} + \frac{(N+\lambda)^2 - \frac{1}{4}}{1-x^{2}} \;,
\label{eq:JacobiODEtrafo}
\end{equation}
has the solutions $u(x)=(1-x^2)^{(2\lambda+1)/4} P^{(\lambda)}_{N}(x)$, see \cite[(4.7.10)]{Szegoe1978}.

For the convenience of the reader we recall next some standard tools that are used to estimate the zeros of classical orthogonal polynomials and their spacings.

\begin{theorem}[{Sturm Comparison Theorem, cf. \cite[Section 2]{LM1986}, see also \cite[Section 1.82]{Szegoe1978}}]\label{theo:SturmComparison}
Let $[a,b]\subset \R$ an interval and $f:[a,b]\rightarrow \R$ and $F:[a,b]\rightarrow \R$ two continuous functions. Let $y$ and $Y$, respectively, be nontrivial solutions of the homogeneous linear ordinary differential equations
\begin{align}
y''+f(x)y=0
\qquad \text{and} \qquad
Y''+F(x)Y=0,\label{eq:SturmODEs}
\end{align}
respectively.
\begin{enumerate}[label=(\roman{*}), ref=(\roman{*})]
  \item\label{it:SCT1} Let $y(a)=Y(a)=0$ and $\lim_{x\rightarrow a+} y'(x)=\lim_{x\rightarrow a+} Y'(x)>0$.
  If $F(x)>f(x)$ for $a<x<b$, then $y(x)>Y(x)$ for $a<x \le c$,
  where $c$ is the smallest zero of $Y$ in the interval $(a,b]$.
  \item\label{it:SCT2} Let $y(a)=Y(a)=0$ and $F(x)>f(x)$ for $a<x<b$.
  Then the smallest zero of $Y$ in $(a,b]$ occurs left of the smallest zero of $y$ in $(a,b]$.
  \item\label{it:SCT3} Under the hypothesis of (ii), the $k$-th zero of $Y$ in $(a,b]$ occurs
  before the $k$-th zero of $y$ to the right of $a$. \endproof
\end{enumerate}
\end{theorem}

One application of the Sturm comparison theorem is the Sturm convexity theorem, where two intervals between zeros of the same function are compared.

\begin{theorem}[{Sturm Convexity Theorem, cf. \cite[Section 2]{LM1986}}]\label{theo:SturmConvexity}
Let $y$ be a nontrivial solution of
\begin{align}
y''+f(x)y=0,
\label{eq:SturmODE}
\end{align}
where $f:[a,b] \rightarrow \R$ is continuous and non-increasing. Then the sequence of zeros of $y(x)$ is convex, i.e., for consecutive zeros $x_{1}$, $x_{2}$, $x_{3} \in [a,b]$ of $y(x)$ we get
\[
x_{2} - x_{1} < x_{3} - x_{2}.
\]
\end{theorem}

We also record the following consequence of a theorem by Markoff, see \cite[Theorem 6.21.1, pp. 121 ff]{Szegoe1978}, that
\begin{align}
\label{eq:movetocenter}
\frac{\partial x_\nu}{\partial \lambda} >0 \quad \text{for} \quad 1 \le \nu \le \floor{\frac{N}{2}}.
\end{align}
In other words, taking the symmetry \eqref{eq:symmetry} into account, we observe that the zeros of the ultraspherical polynomial $P^{(\lambda)}_N$ move towards the center of the interval $[-1,1]$ with increasing parameter $\lambda$.

We recall next for later purposes the main relations between ultraspherical, Legendre and Chebyshev polynomials from \cite[Chapter IV]{Szegoe1978}.

Here we are mainly interested in the \emph{Legendre polynomials} $L_N$ of degree $N \in \N$, which are the ultraspherical polynomials $L_N= {P^{(\frac{1}{2})}_N}$, and their first derivatives. Later we shall use the Chebyshev polynomials of first kind $T_N= c_T(N)P^{(0)}_N$ and of second kind $U_N= c_U(N)P^{(1)}_N$ as a tool.

\begin{figure}
\centering
\includegraphics[width=0.65\linewidth]{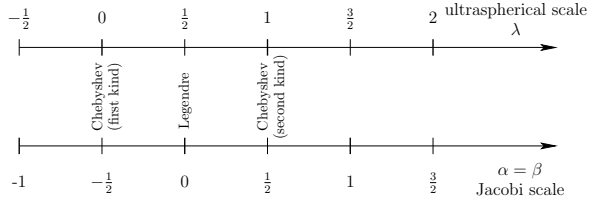}
\caption{Jacobi and ultraspherical scales.}
\label{fig:JacobiGegenbauerScales}
\end{figure}

Our primary interest concerns the special case of \emph{Legendre polynomials}.

\begin{definition}[{Legendre-Gauss-Lobatto nodes, grid and angles; see e.g. \cite[p. 71f, (2.2.18)]{CHQZ2006}}]
For $0 \le k \le N$ the LGL nodes $ \xi^{N}_k$ of \emph{order} $N$ are the $N+1$ zeros of the polynomial $(1-x^2) L'_N(x)$, where $L'_N$ is the first derivative of the Legendre polynomial of {degree} $N$. The LGL nodes are sorted in increasing order, i.e., we have $\xi^{N}_k < \xi^{N}_{k+1}$ for $0 \le k \le N-1$. Their collection $\cGLGL_{N}=\{ \xi_k^N \, : \, 0 \leq k \leq N\}$ forms the LGL grid of order $N$. Given the LGL grid $\cGLGL_N$, we define for $0\le k \le N-1$ the corresponding LGL intervals as $\Delta^{N}_k:=[\xi^{N}_{k},\xi^{N}_{k+1}] \subset [-1,1]$, which are of length $\abs{\Delta^{N}_{k}} = \xi^{N}_{k+1} - \xi^{N}_{k}$.

We also define the corresponding LGL angles $(\eta^{N}_k)_{k=0}^{N} \subset [0,\pi]$ as $\eta^{N}_k=\arccos(-\xi^{N}_k)$.
\end{definition}

\begin{figure}[b]
\begin{center}
\subfloat[LGL grid for odd N.]{\includegraphics[width=0.99\linewidth]{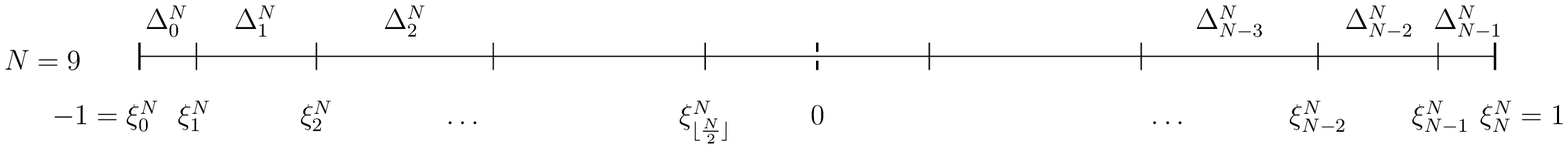}}\\
\subfloat[LGL grid for even N.]{\includegraphics[width=0.99\linewidth]{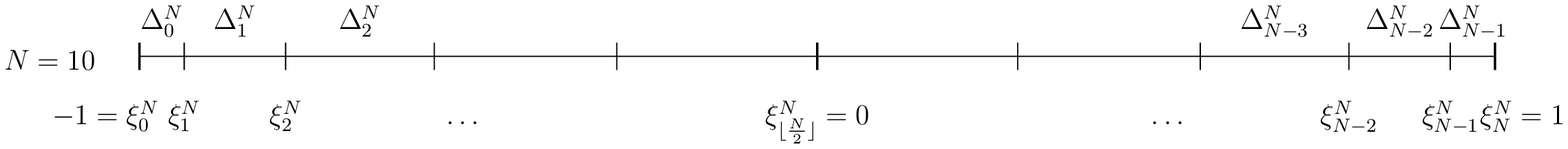}}
\end{center}
\caption{Numeration of LGL nodes and intervals.}
\label{fig:enumeration}
\end{figure}

Our notation deviates from the standard literature on orthogonal polynomials because the LGL {nodes} are usually enumerated in decreasing order, see for example~\cite[Chapter VI]{Szegoe1978}. Note that $\eta^{N}_0=0$ and $\eta^{N}_N=\pi$, i.e. $\xi^{N}_0=-1$ and $\xi^{N}_N=1$. Figure~\ref{fig:enumeration} exemplifies the enumeration scheme used for the LGL points and intervals for odd and even values of $N$.

By the differentiation rule \eqref{eq:diffrule} for ultraspherical polynomials, we have in the special case of Legendre polynomials
\begin{align*}
L_{N}'(x) = \frac{1}{2} P^{(1,1)}_{N-1}(x) = \frac{1}{2} P^{(\frac{3}{2})}_{N-1}(x).
\end{align*}

For the special case of $\lambda=\frac{3}{2}$, the transformed ODE \eqref{eq:JacobiODEtrafo}
\begin{align}
\frac{d^2u}{dx^2} + \phi^{(\frac{3}{2})}_{N}(x) u=0 \qquad \text{with} \quad
\phi^{(\frac{3}{2})}_{N}(x)=\frac{(N+\frac{3}{2})^2-\frac{1}{4}}{1-x^2}
\label{eq:LGLODEtrafo}
\end{align}
has the solution $u(x)=(1-x^2) P^{(\frac{3}{2})}_{N}(x)$, which has its zeros exactly at the LGL nodes of order $N+1$.

By the symmetry relation \eqref{eq:symmetry}, the LGL nodes are symmetric with respect to the origin, i.e., $\xi^{N}_k=-\xi^{N}_{N-k}$ for $0 \le k \le N$, in particular $\xi^{N}_{\frac{N}{2}}=0$ if $N$ is even. For the LGL angles, symmetry yields $\eta^{N}_{N-k}=\pi-\eta^{N}_k$ for $0 \le k \le N$, in particular $\eta^{N}_{\frac{N}{2}}=\frac{\pi}{2}$ if $N$ is even. Consequently, the LGL grids on $[-1,1]$ are \emph{symmetric} around zero and an analogous property holds, of course, for their affine images in general intervals $[a,b]$ with respect to their midpoint $(a+b)/2$. In what follows all grids under consideration are assumed without further mentioning to be symmetric in this sense.

We close this section with recalling a well-known fact about the \emph{monotonicity} of LGL grids since this will be used frequently.

\begin{theorem}[{Monotonicity of LGL interval lengths, cf. \cite[Theorem 5.1 for $\alpha=1$]{JT2009}}]
\label{theo:monotonicitylength}
The LGL interval lengths are strictly increasing from the boundary towards the center of $[-1,1]$, i.e., for $N \ge 3$ and $0\le k \le \floor{\frac{N-1}{2}-1}$ we have $\abs{\Delta^{N}_{k+1}} > \abs{\Delta^{N}_{k}}$.
\end{theorem}
\begin{proof}
By symmetry it suffices to consider the left half of the intervall $[-1,1]$. The LGL points of order $N+1$ are the zeros of the polynomial $u(x)=(1-x^2) P^{(\frac{3}{2})}_{N}(x)$, which is a solution of the transformed ODE \eqref{eq:LGLODEtrafo}. Since $\phi^{(\frac{3}{2})}_{N}(x)$ is monotonically decreasing on $(-1,0]$, by the Sturm Convexity Theorem~\ref{theo:SturmConvexity}, we have
\begin{align*}
\abs{\Delta^{N}_{k+1}} = \xi^{N}_{k+2} - \xi^{N}_{k+1} > \xi^{N}_{k+1} - \xi^{N}_{k} = \abs{\Delta^{N}_{k}},
\end{align*}
which is the assertion. \qed
\end{proof}

Note that for $N=1$ there is only one interval and for $N=2$ the two intervals are of same length by symmetry.


\section{Main Results}\label{sect:main}

In this section we present the main results of this paper. As pointed out in the introduction, they are relevant for the development and analysis of robust preconditioners for high order DG-discretizations. These results can be roughly grouped into two parts, namely (A) results on refined properties of LGL grids themselves and (B) results on the relationship between LGL grids and certain families of \emph{associated} dyadic grids.

For ready use of the subsequent results in the DG context we formulate the results in this section for a generic interval $[a,b]$. As mentioned before LGL nodes and corresponding grids on such an interval are always understood as affine images of the respective quantities on $[-1,1]$ using for simplicity the same notation.

To be precise in what follows one should distinguish between a \emph{grid} and a \emph{partition} or mesh induced by a grid. In general, a grid $\cG$ in $[a,b]$ is always viewed as an ordered set $\{x_j\}_{j=0}^N$ of strictly increasing points $x_j$ - the nodes, with $x_0=a$ and $x_N=b$, which induces a corresponding partition $\cP(\cG)$ formed by the closed intervals $\Delta_j=[x_j,x_{j+1}]$, $0 \le j \le N-1$. The length $x_{j+1}-x_j$ of $\Delta_j$ is denoted by $\fsabs{\Delta_j}$. Conversely, a partition $\cP$ of $[a,b]$ into consecutive intervals induces an ordered grid $\cG(\cP)$ formed by the endpoints of the intervals.

\begin{definition}[Quasi-uniformity]
\label{def:quasi-uniform}
A family of grids $\{\cG_N\}_{N \in \N}$ is called \emph{locally quasi-uniform} if there exists a constant $C_g$ such that
\begin{align}
\label{eq:grid}
C_{g}^{-1} \le \frac{\abs{\Delta }}{\abs{\Delta'}} \le C_{g} \quad \text{for all} \quad \Delta, \Delta'\in \cP(\cG_N),\,\, \Delta \cap \Delta'\neq \emptyset,\,\, N\in\N .\end{align}
\end{definition}

The next notion concerns a certain comparability of two grids.

\begin{definition}[Local uniform equivalence]
\label{def:loc-unif-equiv}
Given two constants $0 < A \leq B$, the grid $\cG$ is said to be \emph{locally} $(A,B)$-\emph{uniformly equivalent} to the grid $\cG'$ if the following condition holds:
\begin{equation}
\label{eq:defeq}
\text{For all} \ \Delta \in \cP(\cG)\;, \quad \Delta' \in \cP(\cG') \;,\quad
\Delta \cap \Delta' \neq \emptyset
~\implies~
A \le {\frac{\abs{\Delta}}{\abs{\Delta'}}} \le B\;.
\end{equation}
\end{definition}


\subsection{Legendre-Gauss-Lobatto grids}\label{sect:LGLgrids}

The main result concerning LGL grids reads as follows.

\begin{theorem}
\label{thm:A}
(i) The family of LGL grids $\{\cGLGL_N\}_{N \in \N}$, is locally quasi-uniform
{and the constant $C_g {=C_g^\text{LGL}}$ in \eqref{eq:grid} satisfies
\begin{align}
\label{eq:Cg-est}
C_g^\text{LGL}\leq \max\left\{1, \frac{7\pi^2}{4}, \frac{3\pi^2}{4}, \frac{9\pi^2}{28}, \frac{486}{65}, \frac{\pi^2}{2}, \frac{49}{8} \right\} = \frac{7\pi^2}{4}.
\end{align}
}
\noindent
(ii) Assume that $M,N \in \N$ with $c\, N \leq M \leq N$ for some fixed constant $c>0$. Then, the LGL grid $\cGLGL_M$ is locally $(A,B)$-uniformly equivalent to the grid $\cGLGL_N$, with $A$ and $B$ depending on $c$ but not on $M$ and $N$.
\end{theorem}

The fact (i) that LGL grids are quasi-uniform seems to be folklore. Since we could not find a suitable reference we restate this fact here as a convenient reference for \cite{BCCD1} and include later a proof with a concrete bound for the constant $C_g$. Claim (ii) is essential for establishing optimality of preconditioners for high order DG discretization with varying polynomial degrees, see \cite{BCCD1}.

\begin{remark}
Numerical experiments show that the quotient of the length of two consecutive LGL intervals is maximized for $\frac{\abs{\Delta^{N}_1}}{\abs{\Delta^{N}_0}}$ and this term grows monotonically in $N$. The value of the (smallest) constant $C_{g}$ is approximately $2.36$.
\end{remark}

A further important issue concerns the comparison of an LGL grid with a ``stretched'' version of itself. To explain this we introduce the \emph{stretching operator} ${L =} L_{[a,b]}:[a,(a+b)/2] \rightarrow [a,b], \, x \mapsto 2x-a$. There is an apparent relation concerning the behavior of LGL grids under stretching which is needed later for establishing subsequent results on associated dyadic grids. The relevant property, illustrated by Figure~\ref{fig:comparisonStretched}, reads as follows.

\begin{definition}
\label{def:S}
A family of symmetric grids $\{\cG_N\}_{N \in \N}$ on $[a,b]$ has property \SbarN for some $\bar N\in \N$, if for any $\cG_{N}$ with $N \le \bar N$ the following holds: For any $I\in \cP(\cG_{N})$ with $I \subset (a,a+(b-a)/4]$ and any $I' \in \cP(\cG_{N})$ such that $L(I)\cap I'\neq \emptyset$, one has $\fsabs{I'} \leq \fsabs{L(I)}$.
\end{definition}

{We have verified numerically that property \SbarN holds for LGL grids up to order $\bar N=2000$. This supports the following conjecture.}

\begin{conjecture}\label{conj:LGLComparisonStretched}
The family of LGL grids $\{\cGLGL_N\}_{N \in \N}$ has property \SbarN for all $\bar N\in \N$.
\end{conjecture}

\begin{figure}[htb]
\centering
\includegraphics[width=0.5\linewidth]{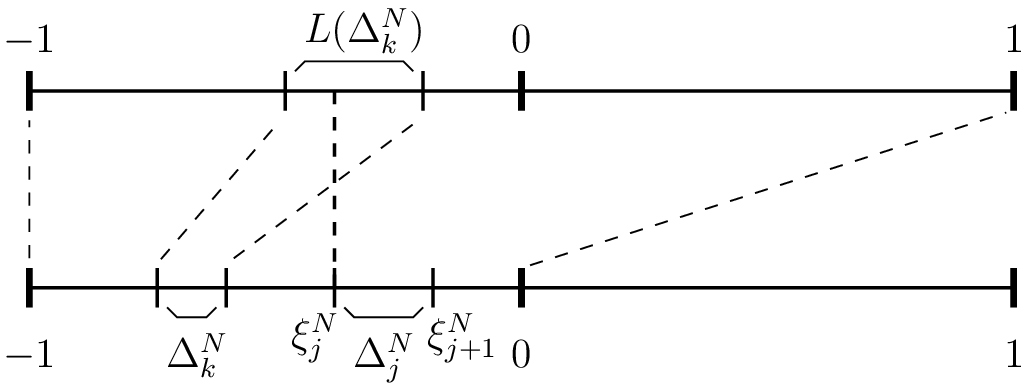}
\caption{Comparison of LGL grid $\cGLGL_N$ with its stretched version.}
\label{fig:comparisonStretched}
\end{figure}


\subsection{Associated dyadic grids}\label{sect:B}

LGL grids are unfortunately not nested, i.e., an interval in $\cP(\cGLGL_N)$ cannot be written as the union of intervals in $\cP(\cGLGL_{N+1})$ which is a severe impediment on the use of LGL grids in the context of preconditioning. Hierarchies of nested grids are conveniently obtained by successive dyadic splits of intervals. More precisely, a single split, i.e., replacing an interval $I$ in a given grid by the two intervals $I', I''$ obtained by subdividing $I$ at its midpoint, gives rise to a \emph{refinement} of the current grid. Hence successive refinements of some initial interval give rise to a sequence of nested grids. We shall see next that that one can associate in a very strong sense with any LGL grid $\cGLGL_N$ a dyadic grid $\cD_N$ that inherits relevant features from $\cGLGL_N$ while in addition the family $\{\cD_N\}_{N\in\N}$ is \emph{nested}.

Moreover, when using the associated dyadic grids {in the context of the Auxiliary Space Method (see \cite{Oswald,Xu,BCCD1}) for preconditioning the systems arising form DG-discretizations, one may obtain meshes with hanging nodes. It then turns out that} it is important that the dyadic grids are ``closed under stretching'' in the following sense.
\begin{definition} [Closedness under stretching]
\label{def:closedness}
Consider the stretching operator $L_{[a,b]}$ introduced above. Then we call a dyadic grid $\cD$ on $[a,b]$
\emph{closed under stretching} if $L_{[a,b]}(\cD \cap [a,(a+b)/2]) \subset \cD$.
\end{definition}

\begin{theorem}
\label{thm:B}
Given any LGL grid $\cGLGL_N$ on any interval $[a,b]$, there exists a dyadic grid $\cD_N$ on $[a,b]$ with the following properties:
\begin{itemize}
\item[(i)]
$\cD_N$ is symmetric and monotonic in the sense of Theorem~\ref{theo:monotonicitylength}, i.e., the intervals $D\in \cP(\cD_N)$ increase in length from left to right in the left half of $[a,b]$.
\item[(ii)]
The family of grids $\{\cD_N\}_{N\in\N}$ is locally quasi-uniform.
\item[(iii)]
The family of grids $\{\cD_N\}_{N\in\N}$ is nested.
\item[(iv)]
For each $N\in\N$ the grid $\cD_N$ is locally $(A,B)$-uniformly equivalent to $\cGLGL_N$, where the constants $A$ and $B$ do not depend on $N$. In particular, one has $N\sim \#(\cD_N)$, uniformly in $N$.
\item[(v)]
Whenever the family of LGL grids $\{\cGLGL_N\}_{N \in \N}$ has property \SbarN for $\bar N \in \N$, then the dyadic grids $\cD_N$ are closed under stretching for $N\leq \bar N$.
\end{itemize}
\end{theorem}

\begin{remark}
\label{rem:stretch}
As mentioned before, \SbarN has been confirmed numerically to hold for $\bar N =2000$, i.e., the dyadic grids $\cD_N$ are closed under stretching for $N\leq 2000$ which covers all cases of practical interest in the context of preconditioning. Conjecture~\ref{conj:LGLComparisonStretched} actually suggests that all dyadic grids $\cD_N$ referred to in Theorem~\ref{thm:B} are closed under stretching.
\end{remark}
For the application of the above results in the context of preconditioning \cite{BCCD1} quantitative refinement of Theorem~\ref{thm:B} (ii) is important. We call a dyadic grid $\cD$ \emph{graded} if any two adjacent intervals $I, I'$ in $\cP(\cD)$ satisfy $\fsabs{I}/\fsabs{I'}\in \{1/2,1,2\}$, i.e., they differ in generation by at most one.
\begin{remark}
\label{rem:graded}
The dyadic grids $\cD_N$ referred to in Theorem~\ref{thm:B} are graded for $N\leq \bar N= 2000$. Moreover, we conjecture that gradedness holds actually for \emph{all} $N\in \N$, see the discussion following Conjectures~\ref{conj:Quotients} and~\ref{conj:QuotientsBessel} in Section~\ref{sect:QuasiuniformityLGL} and Proposition~\ref{prop:graded}.
\end{remark}

The remainder of the paper is essentially devoted to proving Theorems~\ref{thm:A} and~\ref{thm:B}. This requires on the one hand refining our understanding of LGL grids, which, in particular, leads to some results that are perhaps interesting in their own right, and, on the other hand, developing and analyzing a concrete algorithm for creating $\cD_N$ for any given $N$.

Since all the properties claimed in Theorems~\ref{thm:A} and~\ref{thm:B} are invariant under affine transformations it suffices to consider from now on only the reference interval $[a,b]=[-1,1]$.


\section{Proof of Theorem~\ref{thm:A}}\label{sect:ProofA}

In this section we proceed establishing further quantitative facts about LGL grids needed for the proof of Theorem~\ref{thm:A}. We address first interrelations of LGL nodes within a single grid $\cGLGL_N$ of arbitrary order $N$. We begin recalling the following estimates on LGL angles which can be found, for instance, in \cite{CHQZ2006} and \cite{Suendermann1983}
\begin{equation}
\label{eq:LGLa_estimates}
\begin{aligned}
\eta^{N}_{k} \in \left[\ueta^{N}_{k}, \oeta^{N}_{k}\right] \quad \text{for} \quad 1 \le k \le \floor{\frac{N-1}{2}}\\
\quad \text{with} \quad
\ueta^{N}_{k} := \pi \frac{k}{N}
\quad \text{and} \quad
\oeta^{N}_{k}:=\pi \frac{2k+1}{2N+1} .
\end{aligned}
\end{equation}

The lower bound follows from \eqref{eq:movetocenter} and the fact that the ultraspherical polynomial $P^{(1)}_{N-1}$ can, up to a non-zero normalization factor, be identified as the Chebyshev polynomial of second kind $U_{N-1}$, which has the zeros $\cos \ueta^{N}_{k}$ for $1 \le k \le N-1$, see e.g. \cite[(2.3.15) on p. 77]{CHQZ2006}. The upper bound is given in \cite[Lemma 1]{Suendermann1983}.

Note that the lower bound obtained from \eqref{eq:movetocenter} is sharper than the lower estimate in \cite[Lemma 1]{Suendermann1983}.


\subsection{Legendre-Gauss-Lobatto intervals and their spacings}

Next we consider the spacings between two subsequent LGL nodes. To this end, in the following, we will repeatedly make use of the following elementary estimates: for any $x \in [0,\frac{\pi}{2}]$ one has
\begin{equation} \label{eq:estimatetrig}
a) \ \frac{2}{\pi} x \le \sin x \le x \;; \quad
b) \ 1-\frac{2}{\pi} x \le \cos x \le 1 \;; \quad
c) \ 1 -\frac{1}{2}x^2 \le \cos x \le 1-\frac{4}{\pi^2}x^2 \;.
\end{equation}
Moreover, we shall use the prostapheresis trigonometric identity
\begin{equation}
\cos x - \cos y = -2 \sin\left(\frac{x+y}{2}\right) \sin\left(\frac{x-y}{2}\right) \label{eq:TrigDiffCos} \;.
\end{equation}
Note that for $x,y \in [0,\frac{\pi}{2}]$ and $x \ge y$ the arguments of the sine function $\frac{x+y}{2}$ and $\frac{x-y}{2}$ are both within $[0,\frac{\pi}{2}]$.

Recall that $N \ge 1$ as the LGL nodes include the boundary points. For $N=1$ the LGL nodes are $\xi^1_0=-1$ and $\xi^1_1=-1$, for $N = 2$ we have $\xi^2_0=-1$, $\xi^2_1=0$ and $\xi^2_1=-1$. For higher order LGL nodes, unfortunately there is no simple analytical expression. In order to study LGL intervals for larger $N$, we estimate the length of the LGL intervals from below and above using \eqref{eq:LGLa_estimates}. For the proof we consider separately the cases where the position of a LGL node is the endpoint or the center of the interval $[-1,1]$.

\begin{property}[Estimates for LGL interval lengths]\label{lem:LGLeIlengths}
For $N\ge 5$ and $1\le k \le \floor{\frac{N-1}{2}-1}$, we can bound the length of the interval $\Delta^{N}_k$ by
\begin{multline}
2 \sin\left(\frac{\pi}{2} \frac{4kN+k+3N+1}{N(2N+1)} \right) \sin\left(\frac{\pi}{2} \frac{N+k+1}{N(2N+1)} \right)
\le\\
\abs{\Delta^{N}_{k}}
\le 2 \sin\left(\frac{\pi}{2} \frac{4kN+3N+k}{N(2N+1)} \right) \sin\left(\frac{\pi}{2} \frac{3N-k}{N(2N+1)} \right).
\label{eq:LGLlengthGeneral}
\end{multline}
For $N\ge 3$ the length of the boundary interval $\Delta^{N}_0$ can be estimated by
\begin{equation}
\begin{aligned}
\frac{4}{N^2} \le \abs{\Delta^{N}_0} \le \frac{9\pi^2}{2(2N+1)^2}.
\label{eq:LGLlength0}
\end{aligned}
\end{equation}
Moreover, if $N$ is odd and $N\ge 3$, we get
\begin{equation}
\begin{aligned}
\frac{2}{2N+1} \le \abs{\Delta^{N}_{\floor{\frac{N}{2}}}} \le \frac{4N-2}{N^2}\;,
\label{eq:LGLlengthInnerOdd}
\end{aligned}
\end{equation}
whereas if $N$ is even and $N \ge 4$, we have
\begin{equation}
\begin{aligned}
\frac{3}{2N+1} \le \abs{\Delta^{N}_{\frac{N}{2}-1}} \le \frac{4N-4}{N^2}.
\label{eq:LGLlengthInnerEven}
\end{aligned}
\end{equation}
\end{property}
\begin{proof}
From \eqref{eq:LGLa_estimates} for $N\ge 5$ and $1\le k \le \floor{\frac{N-1}{2}-1}$ we can derive the upper estimate
\begin{align*}
\begin{aligned}
\abs{\Delta^{N}_{k}}
&= \xi^{N}_{k+1} - \xi^{N}_{k}
\le -\cos \oeta^{N}_{k+1} + \cos \ueta^{N}_{k}
= -\cos \left( \pi \frac{2k+3}{2N+1}\right) + \cos \left(\pi \frac{k}{N}\right)\\
&\stackrel{\eqref{eq:TrigDiffCos}}{=} -2 \sin\left(\frac{\pi}{2} \left(\frac{k}{N} + \frac{2k+3}{2N+1} \right)\right)
\sin\left(\frac{\pi}{2} \left(\frac{k}{N} - \frac{2k+3}{2N+1} \right)\right) \\
&= 2 \sin\left(\frac{\pi}{2} \frac{4kN+3N+k}{N(2N+1)} \right) \sin\left(\frac{\pi}{2} \frac{3N-k}{N(2N+1)} \right).
\end{aligned}
\end{align*}
Analogously, we have in the same case as lower estimate
\begin{align*}
\begin{aligned}
\abs{\Delta^{N}_{k}}
&= \xi^{N}_{k+1} - \xi^{N}_{k}
\ge -\cos \ueta^{N}_{k+1} + \cos \oeta^{N}_{k}\\
&= -\cos\left(\pi \frac{k+1}{N}\right) + \cos\left(\pi \frac{2k+1}{2N+1}\right)\\
&\stackrel{\eqref{eq:TrigDiffCos}}{=} -2 \sin\left(\frac{\pi}{2} \left(\frac{2k+1}{2N+1} + \frac{k+1}{N}\right)\right)
\sin\left(\frac{\pi}{2} \left(\frac{2k+1}{2N+1} - \frac{k+1}{N}\right)\right)\\
&= 2 \sin\left(\frac{\pi}{2} \frac{4kN+k+3N+1}{N(2N+1)} \right) \sin\left(\frac{\pi}{2} \frac{N+k+1}{N(2N+1)} \right).
\end{aligned}
\end{align*}

In the special case of the boundary interval $\Delta^{N}_0=[\xi^{N}_0,\xi^{N}_1]$ for $N\ge 3$, we have
\begin{align*}
\abs{\Delta^{N}_0}=\xi^{N}_1-\xi^{N}_0=\xi^{N}_1+1=-\cos(\eta^{N}_1)+1.
\end{align*}
Therefore we obtain
\begin{align*}
\abs{\Delta^{N}_0} \le -\cos\left(\pi \frac{3}{2N+1}\right) + 1 \stackrel{(\ref{eq:estimatetrig}~c)}{\le} \frac{9\pi^2}{2(2N+1)^2}
\end{align*}
and
\begin{align*}
\abs{\Delta^{N}_0} \ge -\cos\left(\pi \frac{1}{N}\right) + 1 \stackrel{(\ref{eq:estimatetrig}~b)}{\ge} \frac{4}{N^2}.
\end{align*}

If $N$ is odd ($\floor{\frac{N}{2}}=\frac{N-1}{2}$) and $N \ge 3$, there is a central LGL interval $\Delta^{N}_{\floor{\frac{N}{2}}}$ of size $\abs{\Delta^{N}_{\floor{\frac{N}{2}}}}=-2\xi^{N}_{\floor{\frac{N}{2}}} = 2 \cos(\eta^{N}_{\floor{\frac{N}{2}}})$, see Figure~\ref{fig:enumeration}(a). We estimate its length by
\begin{align*}
\abs{\Delta^{N}_{\floor{\frac{N}{2}}}} \le 2 \cos\left(\frac{\pi}{2} \frac{N-1}{N}\right) \stackrel{(\ref{eq:estimatetrig}~c)}{\le} 2 \left(1-\left(\frac{N-1}{N}\right)^2\right) = \frac{4N-2}{N^2}
\end{align*}
and
\begin{align*}
\abs{\Delta^{N}_{\floor{\frac{N}{2}}}} \stackrel{(\ref{eq:estimatetrig}~b)}{\ge} 2\cos\left(\frac{\pi}{2} \frac{N}{N+\frac{1}{2}}\right) \ge 2 \left(1-\frac{2N}{2N+1}\right) = \frac{2}{2N+1}.
\end{align*}

If $N$ is even ($\floor{\frac{N}{2}}=\frac{N}{2}$) and $N \ge 4$, the center of the interval is a LGL node, i.e., we have $\eta^{N}_{\frac{N}{2}}=\frac{\pi}{2}$, $\xi^{N}_{\frac{N}{2}}=0$, see Figure~\ref{fig:enumeration}(b). The adjacent interval $\Delta^{N}_{\frac{N}{2}-1}$ has the length $\abs{\Delta^{N}_{\frac{N}{2}-1}}=\cos(\eta^{N}_{\frac{N}{2}-1})$. This can be estimated by
\begin{align*}
\abs{\Delta^{N}_{\frac{N}{2}-1}} \stackrel{(\ref{eq:estimatetrig}~c)}{\le} 1-\frac{4}{\pi^2} \left(\pi \frac{\frac{N}{2}-1}{N}\right)^2 = \frac{4N-4}{N^2}
\end{align*}
and
\begin{align*}
\abs{\Delta^{N}_{\frac{N}{2}-1}} \stackrel{(\ref{eq:estimatetrig}~b)}{\ge} 1-\frac{2}{\pi} \left(\pi \frac{\frac{N}{2}-\frac{1}{2}}{N+\frac{1}{2}}\right) = \frac{3}{2N+1},
\end{align*}
which closes the proof. \qed
\end{proof}

\begin{remark}
The estimates for LGL interval lengths in Property~\ref{lem:LGLeIlengths} show that $\abs{\Delta^{N}_{0}} \simeq \frac{1}{N^2}$ for the boundary interval and
$\abs{\Delta^{N}_{\floor{\frac{N-1}{2}}}} \simeq \frac{1}{N}$ for the intervals at the origin.
\end{remark}


\subsection{Quasi-uniformity of Legendre-Gauss-Lobatto grids}\label{sect:QuasiuniformityLGL}

In this subsection we present the proof of Theorem~\ref{thm:A} (i), i.e., we show that two consecutive LGL intervals differ in length at maximum by a constant factor independent of $N$.
\begin{property}[Quasi-uniformity of LGL grids]\label{lem:QuasiUniformity}
The LGL grids $\cGLGL_N$ satisfy
\begin{align*}
C_{g}^{-1} \le \frac{\abs{\Delta^{N}_{k}}}{\abs{\Delta^{N}_{k-1}}} \le C_{g} \quad \text{for all} \quad 1 \le k \le N-1, \
N \geq 2\;,
\end{align*}
where the constant $C_{g} {=C_g^\text{LGL}}$ is bounded by the right hand side in \eqref{eq:Cg-est}.
\end{property}
\begin{proof}
By symmetry, it suffices to consider all subintervals that have a nonzero intersection with the left half interval $[-1,0)$. For $N=2$ the two LGL intervals $[-1,0]$ and $[0,1]$ are of the same size. For $N\ge 3$ we apply Property~\ref{lem:LGLeIlengths} for the quotient of two consecutive interval lengths.

For $2\le k \le \floor{\frac{N-1}{2}-1}$ we have
\begin{align*}
\begin{aligned}
\frac{\abs{\Delta^{N}_{k}}}{\abs{\Delta^{N}_{k-1}}}
&\le \frac{\sin\left( \frac{\pi}{2} \frac{4kN+3N+k}{N(2N+1)} \right)
\sin\left( \frac{\pi}{2} \frac{3N-k}{N(2N+1)}\right)}
{\sin\left( \frac{\pi}{2} \frac{4(k-1)N+k+3N}{N(2N+1)} \right)
\sin\left(\frac{\pi}{2} \frac{N+k}{N(2N+1)} \right)}\\
&\stackrel{(\ref{eq:estimatetrig}~a)}{\le} \left(\frac{\pi}{2}\right)^2
\frac{4kN+3N+k}{4(k-1)N+3N+k} \cdot \frac{3N-k}{N+k}\\
&\le \left(\frac{\pi}{2}\right)^2 \left( 1 + \frac{4N}{4(k-1)N+3N+k} \right) \frac{3N-k}{N+k}
\le \left(\frac{\pi}{2}\right)^2 \cdot \frac{7}{3} \cdot 3 = \frac{7\pi^2}{4}
\end{aligned}
\end{align*}
and
\begin{align*}
\begin{aligned}
\frac{\abs{\Delta^{N}_{k}}}{\abs{\Delta^{N}_{k-1}}}
&\ge \frac{\sin\left(\frac{\pi}{2} \frac{4kN+k+3N+1}{N(2N+1)} \right)
\sin\left(\frac{\pi}{2} \frac{N+k+1}{N(2N+1)} \right)}
{\sin\left(\frac{\pi}{2} \frac{4(k-1)N+3N+k-1}{N(2N+1)} \right)
\sin\left(\frac{\pi}{2} \frac{3N-k+1}{N(2N+1)} \right)}\\
&\stackrel{(\ref{eq:estimatetrig}~a)}{\ge} \left(\frac{2}{\pi}\right)^2 \frac{4kN+3N+k+1}{4(k-1)N+3N+k-1} \cdot \frac{N+k+1}{3N-k+1}\\
&\ge \left(\frac{2}{\pi}\right)^2 \cdot 1 \cdot \frac{1}{3} = \frac{4}{3\pi^2}.
\end{aligned}
\end{align*}

For the quotient including the boundary interval $\Delta^{N}_0$, we have
\begin{align*}
\frac{\abs{\Delta^{N}_1}}{\abs{\Delta^{N}_0}}
&\le \frac{2 \sin\left(\frac{\pi}{2} \frac{7N+1}{N(2N+1)} \right) \sin\left(\frac{\pi}{2} \frac{3N-1}{N(2N+1)} \right)}{\frac{4}{N^2}}\\
&\stackrel{(\ref{eq:estimatetrig}~a)}{\le} \frac{\pi^2}{8} \frac{(7N+1)(3N-1)}{(2N+1)^2} \le \frac{3 \pi^2}{4}
\end{align*}
and
\begin{align*}
\frac{\abs{\Delta^{N}_1}}{\abs{\Delta^{N}_0}}
&\ge \frac{2 \sin\left(\frac{\pi}{2} \frac{7N+2}{N(2N+1)} \right) \sin\left(\frac{\pi}{2} \frac{N+2}{N(2N+1)} \right)}{\pi^2 \frac{9}{2(2N+1)^2}}\\
&\stackrel{(\ref{eq:estimatetrig}~a)}{\ge} \frac{4}{9\pi^2} \frac{(7N+2)(N+2)}{N^2} \ge \frac{28}{9\pi^2}\;.
\end{align*}

If $N$ is even ($\floor{\frac{N}{2}}=\frac{N}{2}$) and $N\ge4$, we can estimate the quotient from above by
\begin{align*}
\frac{\abs{\Delta^{N}_{\floor{\frac{N}{2}}-1}}}{\abs{\Delta^{N}_{\floor{\frac{N}{2}}-2}}}
&\le \frac{\frac{4N-4}{N^2}}{2 \sin\left(\frac{\pi}{2} \frac{4N^2-9N-2}{2N(2N+1)} \right) \sin\left(\frac{\pi}{2} \frac{3N-2}{2N(2N+1)} \right)}\\
&\stackrel{(\ref{eq:estimatetrig}~a)}{\le} \frac{2(4N-4)(2N+1)^2}{(4N^2-9N-2)(3N-2)}
\stackrel{N \ge 4}{\le} \frac{2 \cdot 12 \cdot 9^2}{26 \cdot 10}
= \frac{486}{65} \;,
\end{align*}
because the very but last term is monotonically decreasing for $N\ge4$.
From below we can estimate the quotient by
\begin{align*}
\frac{\abs{\Delta^{N}_{\floor{\frac{N}{2}}-1}}}{\abs{\Delta^{N}_{\floor{\frac{N}{2}}-2}}}
&\ge \frac{\frac{3}{2N+1}}{2 \sin\left(\frac{\pi}{2} \frac{4N^2-9N-4}{2N(2N+1)} \right) \sin\left(\frac{\pi}{2} \frac{5N+4}{2N(2N+1)} \right)}\\
&\stackrel{(\ref{eq:estimatetrig}~a)}{\ge} \frac{6}{\pi^2} \frac{4N^2(2N+1)}{(4N^2-9N-4)(5N+4)}
\ge \frac{6}{\pi^2} \frac{4N^2 \cdot 2N}{4N^2 \cdot 6N}
= \frac{2}{\pi^2}.
\end{align*}

On the other hand, if $N$ is odd ($\floor{\frac{N}{2}}=\frac{N-1}{2}$) and $N\ge 3$, we can estimate the quotient from above by
\begin{align*}
\frac{\abs{\Delta^{N}_{\floor{\frac{N}{2}}}}}{\abs{\Delta^{N}_{\floor{\frac{N}{2}}-1}}}
&\le \frac{\frac{4N-2}{N^2}}{2 \sin\left(\frac{\pi}{2} \frac{4N^2-5N-1}{2N(2N+1)} \right) \sin\left(\frac{\pi}{2} \frac{3N-1}{2N(2N+1)} \right)}\\
&\stackrel{(\ref{eq:estimatetrig}~a)}{\le} \frac{2(4N-2)(2N+1)^2}{(4N^2-5N-1)(3N-1)}
\stackrel{N \ge 3}{\le} \frac{2 \cdot 10 \cdot 7^2}{20 \cdot 8}
= \frac{49}{8}\;,
\end{align*}
because the very but last term is monotonically decreasing for $N\ge3$.
For the estimate from below we have
\begin{align*}
\frac{\abs{\Delta^{N}_{\floor{\frac{N}{2}}}}}{\abs{\Delta^{N}_{\floor{\frac{N}{2}}-1}}}
&\ge \frac{\frac{2}{2N+1}}{2 \sin\left(\frac{\pi}{2} \frac{4N^2-5N-3}{2N(2N+1)} \right) \sin\left(\frac{\pi}{2} \frac{5N+3}{2N(2N+1)} \right) }\\
&\stackrel{(\ref{eq:estimatetrig}~a)}{\ge} \frac{4}{\pi^2} \frac{4N^2(2N+1)}{(4N^2-5N-3) (5N+3)}
\ge \frac{4}{\pi^2} \frac{4N^2\cdot 2N}{4N^2 \cdot 6N}
= \frac{4}{3\pi^2}.
\end{align*}

Overall, we have the desired estimate for $C_g$ given by \eqref{eq:Cg-est}, as claimed.\qed
\end{proof}

For later purposes it is important to have as tight a bound for $C_g$ as possible and the one derived above will not quite be sufficient. The remainder of this section is devoted to a discussion of possible quantitative improvements and related conjectures.

\begin{definition}
\label{def:Quotients}
We say that the family of LGL grids $\{\cGLGL_N\}_{N \in \N}$ has property \MQ for some $\bar N\in \N$, if the quotients
\begin{align}\label{eq:quot-conj}
q_{k}^{N}:=\frac{\abs{\Delta^{N}_{k}}}{\abs{\Delta^{N}_{k-1}}} = \frac{\xi^{N}_{k+1} -\xi^{N}_{k}}{\xi^{N}_{k} -\xi^{N}_{k-1}}
\end{align}
{satisfy
\begin{enumerate}[label=(\roman{*}), ref=(\roman{*})]
  \item $q_k^N \leq q_k^{N+1}$ for $N < \bar N$, i.e., the quotients increase monotonically in $N$ for each fixed $k\in \{1,\ldots,\fsfloor{N/2}\}$,\label{QuotientsBoundaryIncWithN}
  \item $q_k^N \geq q_{k+1}^{N}$ for $1 \le k \le \floor{\frac{N-3}{2}}$, i.e., the quotients decrease monotonically in $k \in \{1,\ldots,\floor{\frac{N-3}{2}}\}$ for any fixed $N\le \bar N$. \label{QuotientsDecWithk}
\end{enumerate}
}
\end{definition}

In the present context property (i) is only relevant for $k=1,2$. The following observations are immediate and recorded for later use.

\begin{remark}
\label{rem:SmallestConstant}
{
Assume that property \MQ holds for some $\bar N \in \N$. Then, the value of the (smallest) constant $C_g^\text{LGL}$ in \eqref{eq:grid} for the family of LGL grids $\{\cGLGL_N\}_{N \leq \bar N}$ is $q_1^{\bar N}$. Moreover, when omitting the outermost intervals, the constant $\tilde C_g$ satisfying
\begin{equation}
\label{eq:inner}
\tilde C_g^{-1}\leq \abs{\frac{\Delta^N_k}{\Delta^N_{k-1}}} \leq \tilde C_g,\quad 2\leq k\leq N-2,\,\, 2<N\leq \bar N,
\end{equation}
is bounded by $q^{\bar N}_2$.
}

Specifically, one can verify numerically that \MQ holds for $\bar N =2000$, in which case one has
\begin{equation}
\label{eq:q-2000}
q_1^{2000} = 2.352303456118672 \, \pm \, 10^{-15},\quad \quad
q_2^{2000} = 1.571697180994308 \, \pm \, 10^{-15}.
\end{equation}
\end{remark}

{
Numerical evidence supports the following
\begin{conjecture}\label{conj:Quotients}
The LGL grids $\cGLGL_N$ have property \MQ ~for all $\bar N\in \N$.
\end{conjecture}
}

In order to determine a constant $C_g^\text{LGL}$ that holds for all $N\in \N$, one can exploit the well known fact
that the asymptotic behavior of the LGL nodes can be expressed by means of the zeros of the Bessel function $J_1$.

\begin{theorem}[{\cite[Theorem 8.1.2]{Szegoe1978}}]\label{theo:LGLasymptotics}
The asymptotic behavior of the LGL angles $(\eta^N_k)_{k=0}^N$ is given by the formula
\begin{align*}
\lim_{N\rightarrow\infty} N \eta^N_{k} = j_{1,k},
\end{align*}
where $j_{1,k}$ are the nonnegative zeros of the Bessel function $J_1(x)$.
\end{theorem}

The first nonnegative zeros of the Bessel function $J_1(x)$ are given in Table~\ref{tab:BesselJ1zeros}.
\begin{table}
\centering
\scriptsize
\begin{tabular}{|c|c||c|c||c|c|}
\hline
$k$ & $j_{1,k}$&$k$ & $j_{1,k}$&$k$ & $j_{1,k}$\\
\hline
0 & 0                     &  4 & 13.323691936314223032 &  8 & 25.903672087618382625\\
1 & 3.8317059702075123156 &  5 & 16.470630050877632813 &  9 & 29.046828534916855067\\
2 & 7.0155866698156187535 &  6 & 19.615858510468242021 & 10 & 32.189679910974403627\\
3 & 10.173468135062722077 &  7 & 22.760084380592771898 & 11 & 35.332307550083865103\\
\hline
\end{tabular}
\caption{The first nonnegative zeros of the Bessel function $J_1$ of first kind. \cite{Reiser2010}}
\label{tab:BesselJ1zeros}
\end{table}

The following arguments support the validity of Conjecture~\ref{conj:Quotients}. For part (ii) we consider again the proof of the Sturm Convexity Theorem~\ref{theo:SturmConvexity} as a motivation and apply the Sturm Comparison Theorem~\ref{theo:SturmComparison} to the present situation. Let $\xi_{0}$, $\xi_{1}$, and $\xi_{2}$ be three consecutive LGL nodes. By definition these points are zeros of the polynomial $u(x)=(1-x^2) P^{(\frac{3}{2})}_{N-1}(x)$, which is a solution of equation \eqref{eq:LGLODEtrafo}. The affine mapping
\begin{align*}
T:x \mapsto \frac{\xi_{1}-\xi_{0}}{\xi_{2}-\xi_{1}} (x-\xi_{1}) + \xi_{0} =:cx-\delta
\end{align*}
with $c=\frac{\xi_{1}-\xi_{0}}{\xi_{2}-\xi_{1}}$ and $\delta=c \xi_{1}-\xi_{0}$ maps the zeros $\xi_{0}$ and $\xi_{1}$ to the zeros $\xi_{1}$ and $\xi_{2}$, respectively. Note that by Theorem~\ref{theo:monotonicitylength}, we have $0<c<1$. Moreover we can estimate $\delta=c \xi_{1}-\xi_{0} > c(\xi_{2}-\xi_{1}) >0$.

\begin{figure}[hb]
\centering
\includegraphics[width=0.6\linewidth]{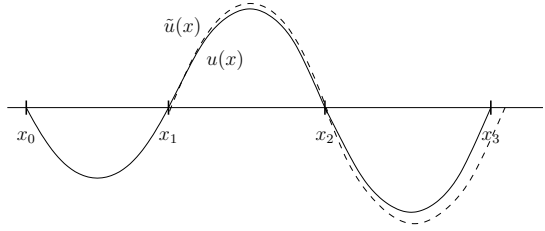}
\caption{The construction for a possible proof of Conjecture~\ref{conj:Quotients}~\ref{QuotientsDecWithk}.}
\label{fig:convexity}
\end{figure}

Now we define $\tilde u:= -u \circ T$, which is a stretched and moved version of $u$ mirrored at the $x$-axis, see Figure~\ref{fig:convexity}. The polynomial $\tilde u$ is a solution of the ODE $\tilde u'' + c^2 \phi^{(\frac{3}{2})}_{N-1}(cx-\delta) =0$. To complete the proof using the Sturm Comparison Theorem~\ref{theo:SturmComparison}, we need to show that
\begin{align*}
& \quad c^2 \phi^{(\frac{3}{2})}_{N-1}(cx-\delta) < \phi^{(\frac{3}{2})}_{N-1}(x),
\end{align*}
which is equivalent to
\begin{align*}
\frac{c^2 c_0}{1-(cx-\delta)^2} < \frac{c_0}{1-x^2} &\quad \text{with} \quad c_0=(N+\frac{1}{2})^2-\frac{1}{4}>0.
\end{align*}
By elementary calculations this can be shown to be equivalent to $c^2 + \delta^2 - 1 < x 2\delta c$.

\begin{remark}
\label{rem:conj}
Since we only need to consider $x \in (\xi_{2},\xi_{3})$, Conjecture~\ref{conj:Quotients}~(ii) would follow from the inequality
\begin{equation}
\frac{c^2 + \delta^2 - 1}{2\delta c} < \xi_{2}\;, \label{eq:ConditionToProve}
\end{equation}
whose proof is still open. However, this inequality has been verified numerically for $N\le \bar N = 2000$.
\end{remark}

A sharper estimate could be obtained in the proof of Theorem~\ref{theo:monotonicitylength} by stretching the mirrored function in order to estimate the stretching constant.

If Conjecture~\ref{conj:Quotients} was correct, the minimal constant in Theorem~\ref{thm:A} (i) could be determined.

\begin{proposition}[M. E. Muldoon, private communication]
\label{prop:Muldoon}
Assume that Conjecture~\ref{conj:Quotients} is true. Then, the value of the (smallest) constant $C_g^\text{LGL}$ in \eqref{eq:grid} for the family of LGL grids $\{\cGLGL_N\}_{N\in\N}$ is $\hat{q}_1 := (j_{1,2}/j_{1,1})^2 -1 = 2.352306 \pm 10^{-6}$.

Moreover, when omitting the outermost intervals the (smallest) constant $\tilde C_g$ satisfying
\begin{equation}
\label{eq:inner2}
\tilde C_g^{-1}\leq \abs{\frac{\Delta^N_k}{\Delta^N_{k-1}}} \leq \tilde C_g,\quad 2\leq k\leq N-2,\,\, N>2,
\end{equation}
is $\hat{q}_2 := (j_{1,3}^2 - j_{1,2}^2)/(j_{1,2}^2- j_{1,1}^2) = 1.571700 \pm 10^{-6}$.
\end{proposition}

\begin{proof}
Since $\eta^{N}_{k}\ \to 0$ as $N \to \infty$, we apply the Taylor expansion
$\xi^{N}_{k}=-\cos \eta^{N}_{k} =-1+\frac{(\eta^{N}_{k} )^2}{2} + \mathcal{O}((\eta^{N}_{k} )^4)$.
Using Theorem~\ref{theo:LGLasymptotics}, we have for each $k$
\begin{equation}
\label{eq:qk}
\frac{\abs{\Delta^{N}_{k}}}{\abs{\Delta^{N}_{k-1}}} {=} \frac{\xi^{N}_{k+1} -\xi^{N}_{k}}{\xi^{N}_{k} -\xi^{N}_{k-1}}
\stackrel{N \rightarrow \infty}{\longrightarrow}
\frac{j_{1,k+1}^2 -j_{1,k}^2}{j_{1,k}^2 -j_{1,k-1}^2}.
\end{equation}
If Conjecture~\ref{conj:Quotients} holds we conclude that
\begin{equation}
\label{eq:moni}
\frac{\abs{\Delta^N_{m}}}{\abs{\Delta^N_{m-1}}}\leq \frac{\abs{\Delta^N_{k}}}{\abs{\Delta^N_{k-1}}} \leq \frac{j_{1,k+1}^2 -j_{1,k}^2}{j_{1,k}^2 -j_{1,k-1}^2}, \quad 1\leq k\leq m \leq N/2,\,N\in \N.
\end{equation}
Thus, taking $k=1$ in \eqref{eq:moni}, one obtains $C_g^\text{LGL}\leq
\frac{j_{1,2}^2 -j_{1,1}^2}{j_{1,1}^2 -j_{1,0}^2} {=} \frac{j_{1,2}^2}{j_{1,1}^2}-1 = \hat q_1$, which confirms the first part of the claim. Likewise, taking $k=2$, we infer that
\begin{align*}
\fsabs{\Delta^N_m/\Delta^N_{m-1}} \leq \fsabs{\Delta^N_2/\Delta^N_1} \leq (j_{1,3}^2 - j_{1,2}^2)/(j_{1,2}^2- j_{1,1}^2)= \hat q_2
\end{align*}
for $m \geq 2$, which finishes the proof. \qed
\end{proof}

Similar ideas like the ones preceding Proposition~\ref{prop:Muldoon} lead us to formulate the following conjecture which is again supported by numerical experiments, see Table~\ref{tab:BesselJ1zerosqrquotients}.

\begin{conjecture}\label{conj:QuotientsBessel}
The quotients
\begin{align*}
\hat{q}_{k}:= \frac{j_{1,k+1}^2-j_{1,k}^2}{j_{1,k}^2-j_{1,k-1}^2}
\end{align*}
decrease monotonically when $k \in \N$ increases.
\end{conjecture}

\begin{table}[ht]
\centering
\scriptsize
\begin{tabular}{|c|c||c|c||c|c|}
\hline
$k$ & $\hat{q}_k$ &
$k$ & $\hat{q}_k$ &
$k$ & $\hat{q}_k$\\
\hline
1 &  2.352305866930589 & 5 &  1.210528759973443 & 9 &  1.114285842810397 \\
2 &  1.571700087758225 & 6 &  1.173914021641693 &10 &  1.102564178478225 \\
3 &  1.363668985974650 & 7 &  1.148148599944429 &11 &  1.093023302970354 \\
4 &  1.266674201821978 & 8 &  1.129032489668108 &12 &  1.085106413514218 \\
\hline
\end{tabular}
\caption{The first {$12$} quotients $\hat{q}_k=\frac{j_{1,k+1}^2-j_{1,k}^2}{j_{1,k}^2-j_{1,k-1}^2}$ , where $j_{1,k}$
are the nonnegative zeros of the Bessel function $J_1$ of first kind.}
\label{tab:BesselJ1zerosqrquotients}
\end{table}


\subsection{Dependence of Legendre-Gauss-Lobatto interval lengths on the order}\label{ssec:lengths}

In this subsection we analyze the behavior of LGL interval lengths with increasing order. The situation in the following theorem is depicted in Figure~\ref{fig:displacement}. The result is an essential ingredient of the proof of Theorem~\ref{thm:B} (iv).

\begin{figure}[b]
\centering
\includegraphics[width=0.7\linewidth]{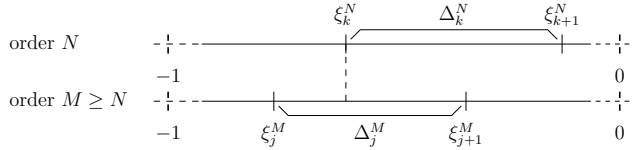}
\caption{Displacement of the LGL intervals with increasing order.}
\label{fig:displacement}
\end{figure}

\begin{theorem}[Displacement of LGL intervals with increasing order]\label{theo:DispLGL2}
Let $N \in \N$ with $N \ge 2$ and $0\le k \le \floor{\frac{N-1}{2}}$. If $\xi^{M}_{j} \le \xi^{N}_{k}$ for $M \in \N$ with $M\ge N$ and $0\le j \le \floor{\frac{M-1}{2}}$, then for the corresponding LGL intervals $\Delta^{M}_{j}=[\xi^{M}_{j},\xi^{M}_{j+1}]$ and $\Delta^{N}_{k}=[\xi^{N}_{k},\xi^{N}_{k+1}]$, we have the inequality $\abs{\Delta^{M}_{j}} \le \abs{\Delta^{N}_{k}}$.
\end{theorem}

\begin{proof}
By symmetry it is sufficient to consider only the left half of $[-1,1]$. If $N=M$, the assertion follows directly from Theorem~\ref{theo:monotonicitylength}. Otherwise we observe that the function
\begin{align*}
\phi^{(\frac{3}{2})}_{N-1}(x)=\frac{(N+\frac{3}{2})^2-\frac{1}{4}}{1-x^2},
\end{align*}
see \eqref{eq:LGLODEtrafo}, is non-increasing on $(-1,0]$. Furthermore $\phi^{(\frac{3}{2})}_{N-1}(x)<\phi^{(\frac{3}{2})}_{M-1}(x)$ for $N<M$ and $x \in (-1,1)$.

Let $y(x)$ be a solution of the ODE $y''(x) + \phi^{(\frac{3}{2})}_{M-1}(x) y(x)=0$, having the zeros $\xi^{M}_{k}$ for $1 \le k \le \floor{\frac{M-1}{2}}$. Then $y(x-\delta)$ is a solution of the ODE $y''(x) + \phi^{(\frac{3}{2})}_{M-1}(x-\delta) y(x)=0$ with zeros $\xi^{M}_{k}+\delta$ for $1 \le k \le \floor{\frac{M-1}{2}}$.

Now we choose $\delta:= \xi^{N}_{k}-\xi^{M}_{j}>0$ such that $\xi^{N}_{k}$ and $\xi^{M}_{j}+\delta$ coincide. Then, by the properties noted above and because of $\xi^{M}_{j} \le \xi^{N}_{k}$, we have
\begin{equation}
\phi^{(\frac{3}{2})}_{N-1}(x) < \phi^{(\frac{3}{2})}_{M-1}(x) \le \phi^{(\frac{3}{2})}_{M-1}(x-\delta).
\end{equation}
By the Sturm comparison theorem, the first zero $\xi^{M}_{j+1}+\delta$ of $(1-x^2) P^{(\frac{3}{2})}_{M-1}(x+\delta)$ to the right of $\xi^{N}_{j} = \xi^{M}_{j}+\delta$ occurs before the first zero of $(1-x^2) P^{(\frac{3}{2})}_{N-1}(x)$ to the right of $\xi^{N}_{j}$, i.e.,
\begin{equation}
\abs{\Delta^{M}_{j}} = \xi^{M}_{j+1}-\xi^{M}_{j} = (\xi^{M}_{j+1}+\delta) - (\xi^{M}_{j}+\delta) \le \xi^{N}_{k+1}-\xi^{N}_{k} = \abs{\Delta^{N}_{k}},
\end{equation}
which completes the proof. \qed
\end{proof}


\subsection{Chebyshev-Gauss-Lobatto nodes and intervals}\label{sect:CGL}
In order to prove Theorem~\ref{thm:A}~(ii) we make a small digression establishing first corresponding properties for another class of grids associated with ultraspherical polynomials, namely \emph{Chebyshev-Gauss-Lobatto (CGL) grids}.
This is a much easier task since the corresponding nodes can be expressed by explicit formulae.
{Recall that the Chebyshev polynomials of first kind $T_N$ of degree $N \in \N$ are also special instances of ultraspherical polynomials. In fact, for $\lambda=0$ one has $T_N=c_T(N) P^{(0)}_N$ with appropriate nonzero normalization constants $c_T(N)$.}

\begin{definition}[Chebyshev-Gauss-Lobatto nodes, grid and intervals]
\label{def:CGL}
For $N \in \N$ we define the CGL nodes $\zeta^N_k$ by
\begin{equation}
\zeta^N_k = -\cos \theta^{N}_k \quad \textnormal{with} \quad
\theta^{N}_k = \frac{\pi k}{N} \quad \textnormal{for} \quad 0 \le k \le N \;.
\end{equation}
Their collection $\cGCGL_N=\{ \zeta_k^N \, : \, 0 \leq k \leq N\}$ forms the CGL grid of order $N$. We also define the corresponding CGL intervals $\Lambda^N_k:=[\zeta^N_{k},\zeta^N_{k+1}]$ for $0 \le k \le N-1$.
\end{definition}
We recall that the points $\zeta_k^N$ with indices $1 \leq k \leq N-1$ are the local extrema of the Chebyshev polynomial of the first kind $T_{N}$, namely the zeros of the Chebyshev polynomial of the second kind $U_{N-1}$.

We can immediately calculate the lengths of the intervals bounded by two consecutive CGL points using the prostapheresis trigonometric identity \eqref{eq:TrigDiffCos}.
Note that for $x,y \in [0,\frac{\pi}{2}]$ and $x \ge y$ the arguments of the sine function $\frac{x+y}{2}$ and $\frac{x-y}{2}$ are both within $[0,\frac{\pi}{2}]$.
The following properties whose analogs for LGL grids have already been established before are simpler in this case but will be needed below.

\begin{property}[CGL interval lengths]
\label{prop:CGLIlengths}
The length of the $k$-th CGL interval $\Lambda^N_k$ is
\begin{align*}
\begin{aligned}
\abs{\Lambda^N_k} = \zeta^N_{k+1}-\zeta^N_{k} = -\cos(\pi \frac{k+1}{N}) + \cos(\pi \frac{k}{N})
= 2 \sin( \pi\frac{2k+1}{2N}) \sin(\frac{\pi}{2N}).
\end{aligned}
\end{align*}
\end{property}

We derive next two types of facts about CGL nodes and corresponding intervals, namely first \emph{monotonicity} statements for any given polynomial degree and second the evolution of interval lengths for increasing degrees.

\subsubsection{Monotonicity and quasi-uniformity properties of Chebyshev-Gauss-Lobatto intervals}

Let us consider the quotient of the lengths of two consecutive CGL intervals
\begin{align}
Q^{N}_{k}:=\frac{\abs{\Lambda^{N}_{k}}}{\abs{\Lambda^{N}_{k-1}}}
= \frac{\zeta^{N}_{k+1}-\zeta^{N}_{k}}{\zeta^{N}_{k}-\zeta^{N}_{k-1}}
= \frac{-\cos(\pi \frac{k+1}{N}) + \cos(\pi \frac{k}{N})}{-\cos(\pi \frac{k}{N}) + \cos(\pi \frac{k-1}{N})}
= \frac{\sin(\pi\frac{2k+1}{2N})}{\sin(\pi\frac{2k-1}{2N})}
\label{eq:quotientCGL}
\end{align}
for $1 \le k \le N-1$.

\begin{property}[Monotonicity of CGL interval lengths, quasi-uniformity of the CGL grid]
\label{property:CGL-uniform}
The lengths of the CGL intervals increase monotonically in the left half of the interval $[-1,1]$, i.e.,
\begin{align*}
\abs{\Lambda^{N}_{k-1}} \le \abs{\Lambda^{N}_{k}} \qquad \textnormal{for} \qquad 0\le k\le \floor{\frac{N-1}{2}}.
\end{align*}
Furthermore, the CGL nodes form a quasi-uniform decomposition of the interval $[-1,1]$, i.e., we have
\begin{align*}
\frac{1}{C} \le \frac{\abs{\Lambda^{N}_{k}}}{\abs{\Lambda^{N}_{k-1}}} \le C
\end{align*}
with $C=\frac{3}{2}\pi$ for $1 \le k \le N-1$.
\end{property}
\begin{proof}
By symmetry \eqref{eq:symmetry}, it is sufficient to consider the left half of $[-1,1]$, i.e., the CGL intervals that have nonempty intersection with $[-1,0)$. If $N$ is even, $\zeta^{N}_{\frac{N}{2}}=0$ is a zero of the Chebyshev polynomial $U_{n-1}$. Therefore, in this case $\Lambda^{N}_{k} \subset [-1,0]$ if and only if $0\le k\le \frac{N-2}{2}$.
Otherwise $N$ is odd and there is additionally a central interval $\Lambda^{N}_{\frac{N-1}{2}} =[\zeta^{N}_{\frac{N-1}{2}}, \zeta^{N}_{\frac{N+1}{2}}]$ that is symmetric with respect to the origin, i.e., the CGL interval $\Lambda^{N}_{k}$ has nonempty intersection with $[-1,0)$ if and only if $0\le k\le \frac{N-1}{2}$. Combining both cases, we conclude that the CGL interval $\Lambda^{N}_{k}$ has nonempty intersection with $[-1,0)$ if and only if $0\le k\le \floor{\frac{N-1}{2}}$.
In this case, the arguments of the sine functions in the numerator and the denominator of the last term of \eqref{eq:quotientCGL} are in $[0,\frac{\pi}{2}]$ and since the sine function is monotonically increasing on $[0,\frac{\pi}{2}]$, we have $1 \le Q^{N}_{k}$, which is the first part of the assertion.

To estimate $Q^{N}_{k}$ from above, we bound the sine function on $[0,\frac{\pi}{2}]$ using (\ref{eq:estimatetrig}~c), and obtain
\begin{align*}
Q^{N}_{k} \stackrel{\eqref{eq:quotientCGL}}{=} \frac{\sin(\pi\frac{2k+1}{2N})}{\sin(\pi\frac{2k-1}{2N})}
\le \frac{\pi\frac{2k+1}{2N}}{\frac{2}{\pi} \pi\frac{2k-1}{2N}}
= \frac{\pi}{2} \frac{2k+1}{2k-1} \le \frac{3}{2} \pi,
\end{align*}
which is the second part of our claims. \qed
\end{proof}

\subsubsection{Displacement of intervals for Chebyshev nodes for increasing order}

As the degree $N$ increases, the CGL intervals decrease in length and move towards the end points of $[-1,1]$. To quantify both effects, we formulate the following proposition on the monotonicity of the interval lengths. As before, by symmetry, we can restrict ourselves to the left half of the interval $[-1,1]$. The following proposition is not strictly needed for our present purposes. But since CGL grids are also used in numerical analysis their association with dyadic grids may also be of interest so that we pause including the following result for completeness.

\begin{proposition}[Displacement of CGL intervals]\label{lem:DispCGL}
Let $N \in \N$, $m \in \N$ and $0\le k\le \floor{\frac{N-1}{2}}$ be given. If for $j\in \N$ we have $\zeta^{N+m}_j \le \zeta^{N}_k$, then the corresponding CGL intervals satisfy $\abs{\Lambda^{N+m}_{j}} \le \abs{\Lambda^{N}_{k}}$.
\end{proposition}
\begin{proof}
Using the assumption on the location of the CGL nodes and the monotonicity properties of the cosine function on $[0,\frac{\pi}{2}]$, we note that $\zeta^{N+m}_j < \zeta^N_k$ is equivalent to
\begin{align*}
\cos \left(\frac{\pi j}{N+m}\right) < \cos\left(\frac{\pi k}{N}\right),
\end{align*}
which, in turn, is equivalent to $ \frac{j}{N+m} > \frac{k}{N}$.
Since the sine function increases monotonically on $[0,\frac{\pi}{2}]$, we can conclude for the lengths of the corresponding CGL intervals, upon using \eqref{eq:TrigDiffCos} several times,
\begin{align*}
\abs{\Lambda^{N+m}_{j}}
&= \zeta^{N+m}_{j+1}-\zeta^{N+m}_{j}
 = -\cos\left(\pi \frac{j+1}{N+m}\right) + \cos\left(\pi \frac{j}{N+m}\right)\nonumber\\
&\stackrel{\eqref{eq:TrigDiffCos}}{=} 2 \sin\left( \frac{\pi(2j+1)}{2(N+m)}\right) \sin\left(\frac{\pi}{2(N+m)}\right)\nonumber\\
&\le 2 \sin\left( \frac{\pi (2 \frac{N+m}{N} k+1)}{2(N+m)}\right) \sin\left(\frac{\pi}{2(N+m)}\right)\nonumber\\
&= 2 \sin\left( \frac{\pi}{2} ( \frac{2k}{N} + \frac{1}{N+m} )\right) \sin\left(\frac{\pi}{2(N+m)}\right)\nonumber\\
&\le 2 \sin\left( \pi\frac{2k+1}{2N}\right) \sin\left(\frac{\pi}{2N}\right)\nonumber\\
&\stackrel{\eqref{eq:TrigDiffCos}}{=} -\cos\left(\pi \frac{k+1}{N}\right) + \cos\left(\pi \frac{k}{N}\right)
= \zeta^{N}_{k+1}-\zeta^{N}_{k} = \abs{\Lambda^{N}_{k}},
\end{align*}
which is the desired estimate. \qed
\end{proof}


\subsection{Locally uniform equivalence of grids}\label{sect:loc-unif-equiv}

We are now prepared to establish locally uniform equivalence of grids of comparable order, associated with a family of ultraspherical polynomials, namely for the Chebyshev and Legendre cases. Here the Chebyshev case will serve as a major tool for deriving later an analog for the Legendre case.

\begin{theorem}\label{prop:lgl-uniequiv-cheb}
Assume that $M, N \in \N$, with $c\, N \leq M \leq N$, for some fixed constant $c>0$. Then, the CGL grid $\cGCGL_M$ in the interval $[-1,1]$ is locally $(A,B)$-uniformly equivalent to the grid $\cGCGL_M$, with $A$ and $B$ depending on $c$ but not on $M$ and $N$.
\end{theorem}
\begin{proof}
Recall from Definition~\ref{def:CGL} that the nodes of the CGL grid of order $N$ are defined as $\zeta_k^N=-\cos \theta_k^N=-\cos \frac{k\pi}N$ for $0 \le k \le N$; $\Lambda_k^N=[\zeta_{k}^N, \zeta_{k+1}^N]$ is the $k$-th interval of this grid, with $0 \leq k \leq N-1$, whose length is given by Property~\ref{prop:CGLIlengths}.

Suppose that $\Lambda_k^N \cap \Lambda_\ell^M \not = \emptyset$, i.e., there exists $x$ such that
\begin{align*}
x=-\cos \theta_x^N =-\cos \frac{k_x \pi}N \qquad \text{for some \ } k_x \in [k,k+1] \;,
\end{align*}
and
\begin{align*}
x=-\cos \theta_x^M =-\cos \frac{\ell_x \pi}M \qquad \text{for some \ } \ell_x \in [\ell,\ell+1] \;.
\end{align*}
The invertibility of the cosine function in $[0,\pi]$ yields
\begin{align*}
\frac{k_x \pi}N = \frac{\ell_x \pi}M
\end{align*}
i.e.,
\begin{align*}
\ell_x= \rho k_x \;, \qquad \text{or} \qquad k_x = \rho^{-1}\ell_x \qquad \text{with} \quad \rho=\frac{M}{N} \in [c,1] \;.
\end{align*}
Then, $\ell \leq \ell_x = \rho k_x \leq \rho (k+1)$ and $k \leq k_x = \rho^{-1} \ell_x \leq \rho^{-1} (\ell+1)$, i.e,
\begin{equation}\label{eq:zzz.2}
\ell \leq \rho (k +1) \;, \qquad k \leq \rho^{-1}(\ell +1) \;.
\end{equation}
Now, using Property~\ref{prop:CGLIlengths} and $\sin x \simeq x$ for small $x$, as well as $M \simeq N$, we have
\begin{align*}
\fsabs{\Lambda_k^N} \simeq \frac{1}{N} \sin \frac{(k+1/2)\pi}N \;, \qquad \text{and} \qquad
\fsabs{\Lambda_\ell^M} \simeq \frac{1}{N} \sin \frac{(\ell+1/2)\pi}M \;.
\end{align*}
Therefore it is enough to prove the uniform equivalence of the two sines. To this end, we first observe that, by \eqref{eq:zzz.2}, we have
\begin{align*}
(\ell+1/2)\frac{\pi}{M} \leq \frac{\rho (k+1) + 1/2}\rho \frac{\pi}{N} = (k+1/2)\frac{\pi}{N} + \frac{1}{2} \left(1+\rho^{-1} \right) \frac{\pi}{N} \;,
\end{align*}
whence, noting that $k \geq 0$, we get
\begin{equation}\label{eq:zzz.3}
\frac{(\ell+1/2)\frac{\pi}{M}}{(k+1/2)\frac{\pi}{N}} \leq 1 + \frac{1}{2} \frac{1+\rho^{-1}}{k+1/2} \leq 2+\rho^{-1} \;.
\end{equation}
Exchanging the roles of $k$ and $\ell$, we obtain
\begin{align*}
\frac{(k+1/2)\frac{\pi}{M}}{(\ell+1/2)\frac{\pi}{N}} \leq 2+\rho \;,
\end{align*}
whence
\begin{equation}\label{eq:zzz.4}
\frac1{2+\rho} \leq \frac{(\ell+1/2)\frac{\pi}{M}}{(k+1/2)\frac{\pi}{N}} \;.
\end{equation}
Putting together \eqref{eq:zzz.3} and \eqref{eq:zzz.4}, we see that we have two angles $\gamma$ and $\varphi$ , that we can assume to be in $(0,\pi/2]$, which satisfy $0 < a \leq (\gamma / \varphi) \leq b$ for some constants $a< 1$ and $b>1$. This immediately implies that there exist constants $a^*$ and $b^*$ with the same properties such that
\begin{align*}
0 < a^* \leq \frac {\sin \gamma} {\sin \varphi} \leq b^* \; .
\end{align*}
Indeed, fix any $\lambda \in (0,1)$ and let $C>0$ such that $Cx \leq \sin x \leq x$ for all $0 \leq x \leq \lambda \pi /2$; then, let $\mu < \lambda$ to be determined in a moment, and let $D>0$ be such that $D \leq \sin x \leq 1$ for all $\mu \pi/2 \leq x \leq \pi/2$. Now, if $b\varphi \leq \lambda \pi/2$, then also $\varphi \leq \lambda \pi/2$ as well as $\gamma \leq b \varphi \leq \lambda \pi/2$, whence
\begin{align*}
\frac {\sin \gamma} {\sin \varphi} \leq \frac { \gamma}{C \varphi} \leq \frac b C \;, \qquad \text{and} \qquad
\frac {\sin \gamma}{\sin \varphi} \geq \frac {C \gamma}{ \varphi} \geq C a \;.
\end{align*}
Conversely, if $b\varphi > \lambda \pi/2$, then $\varphi > (\lambda /b)\pi/2 $ and $\gamma \geq a \varphi \geq (a/b)\lambda \pi/2$. Choosing $\mu=(a/b)\lambda$, we have both $\varphi > \mu \pi/2$ and $\gamma > \mu \pi/2$, whence
\begin{align*}
D \leq \frac {\sin \gamma} {\sin \varphi} \leq \frac 1 D \;.
\end{align*}
This concludes the proof of the theorem. \qed
\end{proof}

In order to establish the analogous property for the LGL grids, we need the following auxiliary result.

\begin{lemma}
\label{lem:CGL}
For any $0 \leq k \leq N-1$, let $\Delta_k^N$ and $\Lambda_k^N$ be, respectively, the $k$-th interval of the LGL grid and the CGL grid of the same order $N$. Then
\begin{align*}
\fsabs{\Delta_k^N} \simeq \fsabs{\Lambda_k^N}
\end{align*}
uniformly in $k$ and $N$.
\end{lemma}
\begin{proof}
For the LGL nodes, by \eqref{eq:LGLa_estimates} we have for $1 \leq k \leq \fsfloor{(N-1)/2}$ that
\begin{align*}
\xi_k^N = -\cos \eta_k^N \;, \qquad \text{with}\quad k \frac \pi N \leq \eta_k^N \leq \frac{2k+1}{2N+1} \pi < (k+1)\frac \pi N\;;
\end{align*}
thus $\zeta_k^N \leq \xi_k^N \leq \zeta_{k+1}^N$, whence $\zeta_k^N \leq \xi_k^N < \xi_{k+1}^N \leq \zeta_{k+2}^N$, i.e.,
\begin{equation}\label{eq:zzz.5}
\Delta_k^N \subset \Lambda_k^N \cup \Lambda_{k+1}^N \;,
\end{equation}
so that $\fsabs{\Delta_k^N} \leq \fsabs{\Lambda_k^N} + \fsabs{\Lambda_{k+1}^N}$. Since we already know that contiguous elements of the same grid have uniformly equivalent lengths, we obtain
\begin{align*}
\fsabs{\Delta_k^N} \lesssim \fsabs{\Lambda_k^N}
\end{align*}
uniformly in $k$ and $N$.

In order to obtain the reverse bound, we use the following lower estimate from Property~\ref{lem:LGLeIlengths}:
\begin{align*}
\fsabs{\Delta_k^N} \geq 2 \sin \left( \frac{\pi}{2} \frac{4kN+k +3N+1}{N(2N+1)} \right) \, \sin \left( \frac{\pi}{2} \frac{N+k+1}{N(2N+1)} \right) \;,
\end{align*}
and we already know from Property~\ref{prop:CGLIlengths} that
\begin{align*}
\fsabs{\Lambda_k^N}=2 \sin \left( \frac{\pi}{2} \frac{2k+1} N \right)\, \sin \left( \frac{\pi}{2} \, \frac{1}{N} \right) \;.
\end{align*}
Now, it is immediate to observe that for the given interval of variation of $k$ we have
\begin{align*}
\frac{1}{2} \leq \frac{N+k+1}{2N+1} \leq 1.
\end{align*}
On the other hand, we write
\begin{align*}
Z:=\frac{4kN+k +3N+1}{2N+1} = \frac{4N+1}{2N+1} \, k + \frac {3N+1} {2N+1}
\end{align*}
and we observe that
\begin{align*}
1 \leq \frac{4N+1}{2N+1} \leq 2 \qquad \text{and} \qquad 1 \leq \frac {3N+1} {2N+1} \leq 2 \;,
\end{align*}
whence
\begin{align*}
\frac{1}{2} (2k+1) \leq Z \leq 2(2k+1) \;.
\end{align*}
Since we have seen in the proof of the previous theorem that uniformly equivalent arguments imply uniformly equivalent sines, we conclude that
\begin{align*}
\fsabs{\Delta_k^N} \gtrsim \fsabs{\Lambda_k^N} \;
\end{align*}
and the proof of the lemma is complete for $1 \leq k \leq \fsfloor{(N-1)/2}$. For $k=0$, we recall that $\zeta^N_0=\xi^N_0$ and $\zeta^N_1 \leq \xi^N_1 < \zeta^N_2$ as seen above. This yields $\Lambda^N_0 \subseteq \Delta^N_0 \subset \Lambda^N_0 \cup \Lambda^N_1$, whence the result for $k=0$. Finally, we observe that intervals are symmetrically placed around the origin. \qed
\end{proof}

The following Corollary finishes the proof of Theorem~\ref{thm:A}~(ii).

\begin{corollary}\label{prop:lgl-uniequiv-leg}
Assume that $c\, N \leq M \leq N$ for some fixed constant $c>0$. Then, the LGL grid $\cGLGL_M$ in the interval $[-1,1]$ is locally $(A,B)$-uniformly equivalent to the grid $\cGLGL_{N}$, with $A$ and $B$ depending on $c$ but not on $M$ and $N$.
\end{corollary}
\begin{proof}
Suppose that $\Delta_k^N \cap \Delta_\ell^M \not = \emptyset$ for some $k$ and $\ell$; then, recalling \eqref{eq:zzz.5}, the set $\Lambda_k^N \cup \Lambda_{k+1}^N $ has a non-empty intersection with the set $\Lambda_\ell^M \cup \Lambda_{\ell+1}^M$, which implies that $\Lambda_m^N \cap \Lambda_n^M \not = \emptyset$ for some $m \in \{k,k+1\}$ and $n \in \{\ell, \ell+1\}$. Using Theorem~\ref{prop:lgl-uniequiv-cheb} and Lemma~\ref{lem:CGL}, as well as the fact that contiguous intervals of any CGL and LGL grid have uniformly comparable lengths (see Property~\ref{property:CGL-uniform}, Theorem~\ref{thm:A} (i)), we obtain
\begin{align*}
\fsabs{\Delta_k^N} \simeq \fsabs{\Lambda_k^N} \simeq \fsabs{\Lambda_m^N} \simeq \fsabs{\Lambda_n^M} \simeq \fsabs{\Lambda_\ell^M} \simeq \fsabs{\Delta_\ell^M} \;,
\end{align*}
which is the claim. \qed
\end{proof}


\section{Dyadic Grid Generation - Proof of Theorem~\ref{thm:B}}\label{sect:dyadic}

The main objective of this section is to construct for the family $\{\cGLGL_N\}_{N\in \N}$, of LGL grids an associated family $\{\cD_N\}_{N\in \N}$, of dyadic grids that are locally $(A,B)$-uniformly equivalent with constants $A,B$, independent of $N$, which in addition offer the additional advantage of being nested. This construction is the essential basis for the proof of Theorem~\ref{thm:B}. Therefore, we formulate the construction for a general interval $[a,b]$. Although we shall be concerned in this section with more general classes of ordered grids $\cG=\{x_i : 0 \le i \le N\} \subset [a,b]$, $x_0=a,x_N=b$, we retain some notation used earlier only for LGL grids such as the notation $\Delta$ for corresponding intervals. We always assume that $\cG$ is symmetric around the midpoint of $[a,b]$ and monotonic in the spirit of Theorem~\ref{theo:monotonicitylength}.
Recall that we denote by {$\cP=\cP(\cG)$ the partition of subintervals of $[a,b]$ induced by $\cG$}; conversely, $\cG=\cG(\cP)$ denotes the grid defined by a partition $\cP$ of $[a,b]$ comprised of the endpoints of its intervals.

The following notions will serve as important tools for the envisaged construction. We exploit the symmetry of $\cG$ and consider for any subinterval $I \subset [a,(a+b)/2]$ the largest and shortest overlapping subcell in $\cP (\cG)$
\begin{equation}
\begin{aligned}
\label{eq:Ou}
\oDelta(I,\cG) &:= \argmax\{\fsabs{\Delta}: \Delta \in \cP (\cG),\, I \cap \Delta \neq \emptyset\} \quad \text{and}\\
\uDelta(I,\cG) &:= \argmin\{\fsabs{\Delta}: \Delta \in \cP (\cG),\, I \cap \Delta \neq \emptyset\}.
\end{aligned}
\end{equation}
Note that due to the monotonicity property of $\cG$ and the restriction of $I$ to one half of the base interval, $\oDelta, \uDelta$ are always uniquely defined.

For any interval $D$ in some dyadic partition $\cP$ {with $D \neq [a,b]$} we shall denote by $\hat D = \hat D(D)$ its parent interval which has $D$ as a subinterval created by splitting $\hat D$ at its midpoint.
The following inequalities will be useful in the sequel:
\begin{equation}\label{eq:ineqDhatD}
\uDelta(\hat{D},\cG) \leq \uDelta({D},\cG) \leq \oDelta({D},\cG) \leq \oDelta(\hat{D},\cG) \;.
\end{equation}
It will be convenient to associate with a dyadic partition $\cP$ of $[a,b]$ the corresponding \emph{rooted binary tree} $T=T(\cP)$ whose nodes are those subintervals generated during the refinement process yielding $\cP$. Thus, $T$ has the interval $[a,b]$ as its \emph{root} and the parent-child relation is given by the inclusion of intervals. Those nodes in the tree that have no children are called \emph{leaves}.
Note that the binary tree $T$ associated with a dyadic grid is a \emph{full} binary tree, i.e., every node has either no or two child nodes in the tree. Furthermore, note that the binary tree $T$ is full if and only if the set of leaves of $T$ is a partition of $[a,b]$.


\subsection[The algorithm Dyadic]{The algorithm \algo{Dyadic}}

Given a real number $\alpha >0$, a grid $\cG$ on the interval $[a,b]$ and an initial dyadic grid $\cD_0$, the following Algorithm~\ref{algo:Dyadic} creates a certain refined dyadic grid $\cD = \algo{DyadicGrid}(\cG, \cD_0, \alpha)$ through suitable successive refinements of partitions.

\begin{algorithm}[ht]
\begin{algorithmic}[1]
\State{$\cP \gets \cP_0:=\cP(\cD_0)$}
\Comment{initialization}
\While{there exists $D\in \cP$ such that $\fsabs{D} > \alpha \fsabs{{\uDelta}(D,\cG)}$ \label{inalgo:split}}
\State{split $D$ by halving into its two children $D', D''$ with $D=D'\cup D''$}
\State{$\cP \gets (\cP\setminus \{D\} )\cup \{D',D''\}$}
\Comment{replace $D$ by $D'$ and $D''$}
\EndWhile
\State{\Return \cD:=\cG(\cP)}
\end{algorithmic}
\caption{Algorithm {$\cD=\algo{Dyadic}(\cG, \cD_0, \alpha)$ for the generation
of dyadic grids.}}
\label{algo:Dyadic}
\end{algorithm}
\noindent
Since the parameter $\alpha$ is usually fixed and clear from the context it will be convenient for further reference, to set
\begin{align*}
\cD = \cD(\cG,\cD_0) =:\algo{DyadicGrid}(\cG, \cD_0, \alpha) \quad \mbox{and}\quad
\cP = \cP(\cG,\cD_0) =: \cP(\cD(\cG,\cD_0)) .
\end{align*}

\smallskip
Of course, a natural ``extreme'' choice for the initial dyadic grid is $\cD_0= \{a,b\}$. Note that clearly $\cP$ is also a dyadic partition of $[a,b]$.

In what follows we shall frequently make use of the following relations which are immediate consequences of the definition of Algorithm~\ref{algo:Dyadic}.

\begin{remark}
\label{rem:relations}
The resulting partition $\cP = {\cP(\cG,\cD_0)}$ has the following properties:
\begin{itemize}
\item[(i)]
For any $\Delta \in {\cP(\cG)} $, $D\in \cP$, one has
\begin{equation}
\label{eq:Ismall}
\Delta \cap D \neq \emptyset ~~\implies ~~
\abs{D} \le \alpha \abs{\Delta}.
\end{equation}
\item[(ii)]
Assume that $\cP_0=\cP(\cD_0)$ satisfies
\begin{equation}
\label{eq:beta}
\fsabs{D} > \beta \fsabs{{\uDelta}(D,\cG)} \quad \text{for all} \ \ D\in \cP_0,
\end{equation}
then, for any $\Delta \in {\cP(\cG)}$, $D\in T(\cP)\setminus \cP$, one has
\begin{equation}
\label{eq:split2}
\fsabs{D} > \min\{\alpha,\beta\} \fsabs{{\uDelta}(D,\cG)}.
\end{equation}
Hence, for any $D \in {\cP(\cG,\{a,b\}) \setminus \{[a,b]\}}$, the parent interval $\hat D$ of length $\fsabs{\hat D}=2 \fsabs{D}$, whose halving produced $D$, satisfies
\begin{equation}
\label{eq:I'big}
\fsabs{\hat D} > \alpha \fsabs{{\uDelta}(\hat D,\cG)} \;.
\end{equation}
\item[(iii)]
Whenever for each $D_0$ from the initial partition $\cP_0 {\setminus \{[a,b]\}}$ {the parent $\hat D$ of $D_0$} satisfies condition \eqref{eq:I'big}, then this is inherited by $\cP= {\cP(\cG,\cD_0)}$, i.e., {for each $D \in \cP\setminus \{[a,b]\}$ the parent $\hat D$ of $D$} also satisfies \eqref{eq:I'big}.
\end{itemize}
\end{remark}

It is easy to see that Algorithm~\ref{algo:Dyadic} terminates after finitely many steps. In fact, the maximal dyadic level of $\cD$ is
\begin{align}\label{eq:finiteMaxLevel}
{ J =\ceil{\log_2\Big( \frac{b-a}{\alpha h}\Big)}, }
\end{align}
where $h := \min\{\abs{\Delta} : \Delta \in {\cP(\cG)} \}$ is the finest resolution in the original grid $\cG$.
Indeed, for any dyadic cell $D$ of maximal level $J$ we have
\begin{align*}
\abs{D} = 2^{-J} (b-a) \le 2^{\log_2(\frac{\alpha h}{b-a})} (b-a) \alpha h \le \alpha \abs{ {\uDelta}(D,\cG)} \;,
\end{align*}
hence, such a cell cannot be halved in the algorithm.

We shall see next {that} $\cD {=\cD(\cG,\cD_0)}=\algo{DyadicGrid}(\cG, \cD_0, \alpha)$ and the original grid $\cG$ are {still locally of comparable size} whenever the initial dyadic grid $\cD_0$ satisfies a condition like \eqref{eq:beta}, but with a potentially smaller constant $\beta$, which will be important later.

\begin{proposition}\label{lem:dyadic_ver2}
Assume that \eqref{eq:grid} holds for a given grid $\cG$ and assume that the initial dyadic partition $\cP_0=\cP(\cD_0)$ satisfies the following condition: there exists some $\beta>0$ such that for all $D\in \cP_0\setminus\{[a,b]\}$, the parent $\hat D= \hat D(D)$ satisfies
\begin{equation}
\label{eq:I'bigg_ver2}
\fsabs{\hat D} > \beta \fsabs{\uDelta(\hat D,\cG)} \;.
\end{equation}
Then, the grid $\cD=\algo{DyadicGrid}(\cG, \cD_0, \alpha)$ is locally $(A,B)$-uniformly equivalent to $\cG$ with $A= \alpha^{-1}$, $B= 2C_g/\min\{\alpha, \beta,1\}$, i.e., one has
\begin{multline}
\label{eq:dyeq_ver2}
\qquad\forall D \in \cP=\cP(\cD), \ \ \forall \Delta \in \cP(\cG)\:,\\
\Delta \cap D \neq \emptyset
~\implies~
\alpha^{-1} \le \frac {\abs{\Delta}}{\abs{D}} \le \frac{2C_g}{\min\{\alpha, \beta,1\}}\;.\qquad
\end{multline}
\end{proposition}
\begin{proof}
The lower bound follows directly from \eqref{eq:Ismall} in Remark~\ref{rem:relations} (i).
We shall show next that there exists an $\eta >0$ such that for any $D\in\cP$
\begin{equation}
\label{eq:toshow_ver2}
\fsabs{D} \geq \eta \fsabs{\Delta} \quad \mbox{holds for all } \Delta \in \cP(\cG) \mbox{ such that } \Delta \cap D\neq \emptyset \;.
\end{equation}
To see this, consider any $D\in\cP\setminus\{[a,b]\}$ and recall from \eqref{eq:split2} in Remark~\ref{rem:relations} (ii) that its parent $\hat D$ satisfies the inequality
\begin{equation}
\label{eq:minab}
\fsabs{\hat D} > \min(\alpha,\beta) \fsabs{\uDelta(\hat D,\cG)} \;.
\end{equation}
Suppose first that $\hat D$ intersects at most two intervals from $\cP(\cG)$; these two intervals therefore have to be $\uDelta(\hat D,\cG), \oDelta(\hat D,\cG)$.
Using \eqref{eq:minab}, \eqref{eq:grid} and \eqref{eq:ineqDhatD}, we obtain
\begin{align*}
2\fsabs{D} =\fsabs{\hat D}
   & > \min(\alpha,\beta) \fsabs{\uDelta(\hat D,\cG)}
     \geq \min(\alpha,\beta) C_g^{-1} \fsabs{\oDelta(\hat D,\cG)} \\
   & \geq \min(\alpha,\beta) C_g^{-1} \fsabs{\oDelta(D, \cG)}
     \geq \min(\alpha,\beta) C_g^{-1} \fsabs{\uDelta(D, \cG)} \;,
\end{align*}
i.e., \eqref{eq:toshow_ver2} holds with $\eta=\min(\alpha,\beta)/(2C_g)$.
Suppose next that one has
$\#\{\Delta \in\cP(\cG): \Delta \cap \hat D \neq \emptyset\} \geq 3 $. Note that, if $\oDelta(\hat D,\cG)$ is not contained in $\hat D$, it must have a neighbor $\Delta'\in \cP(\cG)$ fully contained in $\hat D$. Therefore, by \eqref{eq:grid} and \eqref{eq:ineqDhatD},
\begin{align*}
2\fsabs{D} &= \fsabs{\hat D}
\geq \sum_{\Delta \subset \hat D} \fsabs{\Delta}
\geq \fsabs{\Delta'}
\geq C_g^{-1} \fsabs{\oDelta(\hat D,\cG)}\\
&\geq C_g^{-1} \fsabs{\oDelta(D,\cG)}
\geq C_g^{-1} \fsabs{\uDelta(D,\cG)}\;,
\end{align*}
which means that in this case \eqref{eq:toshow_ver2} is valid with $\eta =1/(2C_g)$. This concludes the proof of \eqref{eq:dyeq_ver2}. \qed
\end{proof}

The value of the parameter $\alpha$ influences the deviation of the cardinality of $\cD=\algo{DyadicGrid}(\cG, \cD_0, \alpha)$ from that of $\cG$: the larger is $\alpha$, the smaller is the number of refinements of the initial grid $\cD_0$ induced by $\cG$. When $\cD_0=\{a,b\}$, a choice of $\alpha$ around $1$ produces comparable cardinalities for $\cD=\algo{DyadicGrid}(\cG, \cD_0, \alpha)$ and $\cG$. This is clearly documented in Figure~\ref{fig:GridsizesDyadic} for the LGL grids $\cGLGL_N$ of increasing order $N$.

\begin{remark}
\label{rem:DyadicAffineInvariance}
Since the algorithm $\algo{Dyadic}(\cG, \cD_0, \alpha)$ only depends on relative sizes of the overlapping subintervals in $\cP(\cG)$ and $\cP$, it is invariant under affine transformations.
\end{remark}

\begin{figure}
\centering
\includegraphics[width=7.6cm]{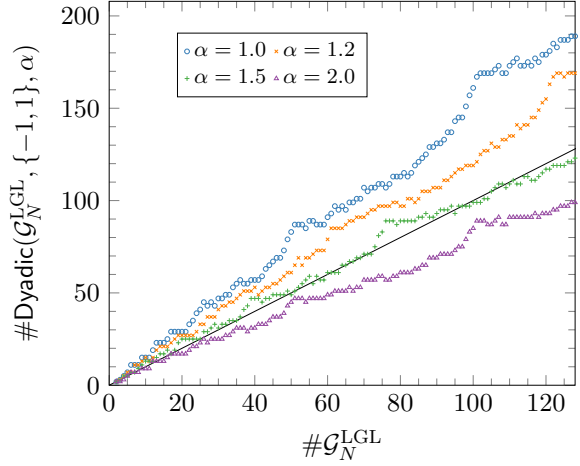}
\caption{Size of the dyadic grid $\algo{DyadicGrid}(\cGLGL_{N}, \{-1,1\}, \alpha)$ as a function of the size of $\cGLGL_{N}$ for different values of $\alpha$.
The black graph indicates the line of equal cardinalities of LGL and dyadic meshes.
}
\label{fig:GridsizesDyadic}
\end{figure}


\subsection{Monotonicity of the dyadic grids}

The grids produced by algorithm $\algo{Dyadic}(\cG, \cD_0, \alpha)$ exhibit two types of monotonicity. We first consider the monotonicity with respect to the parameter $\alpha$ which is obvious.

\begin{remark}[Monotonicity with respect to $\alpha$]
By construction of the algorithm, for any $\alpha, \tilde \alpha \in \R$ with $\alpha < \tilde \alpha$, the dyadic mesh $\algo{DyadicGrid}(\cG, \cD_0, \alpha)$ is equal to or a refinement of the dyadic mesh $\algo{DyadicGrid}(\cG, \cD_0, \tilde \alpha)$.
\end{remark}

We show next that the monotonicity of the interval lengths in the input grid $\cG$ in the sense of Theorem~\ref{theo:monotonicitylength} is inherited by the dyadic grid.

\begin{definition}[Monotonicity] \label{def:monotonicity}
A symmetric grid $\cG$ on $[a,b]$ is \emph{monotonic} if for any $\Delta, \Delta' \in \cP(\cG)$ with $\Delta, \Delta' \subset {[a,(a+b)/2]}$, where $\Delta'$ is the right neighbor of $\Delta$, one has $\abs{\Delta} \le \abs{\Delta'}$.
\end{definition}

\begin{proposition}[Monotonicity of the dyadic grids]
\label{lem:MonotonicityDyadic}
Let $\alpha>0$ and assume that $\cG$ and $\cD_0$ are symmetric and monotonic. Then the dyadic grid $\cD$ generated by $\algo{DyadicGrid}(\cG,\cD_0, \alpha)$ is also symmetric and monotonic.
\end{proposition}
\begin{proof}
Assume that $D, D' \in \cP=\cP(\cD)$ are as in Definition~\ref{def:monotonicity}; suppose by contradiction that $\fsabs{D} > \fsabs{D'}$ which means $\fsabs{D} \geq 2 \fsabs{D'}$. Since $T(\cP)$ is full $D$ is not contained in $\hat D'$, the parent of $D'$, and hence $\uDelta(D',\cG)= \uDelta(\hat D',\cG)$. Since $D$ is a left neighbor of $D'$ so that $\fsabs{\uDelta(D,\cG)} \leq \fsabs{\uDelta(D',\cG)}$ we infer from \eqref{eq:Ismall}
\begin{align*}
\alpha^{-1} \fsabs{D} \leq \fsabs{\uDelta(D,\cG)} \leq \fsabs{\uDelta(D',\cG)} = \fsabs{\uDelta(\hat D',\cG)} < \alpha^{-1} \fsabs{\hat D'}= \alpha^{-1} 2 \fsabs{D'},
\end{align*}
which is a contradiction. \qed
\end{proof}

Using Theorem~\ref{theo:monotonicitylength} we now can conclude that the dyadic grids generated for LGL input grids $\cGLGL_N$ are symmetric and monotonic.

\begin{corollary}
\label{cor:MonotonicityDyadic}
Let $\alpha>0$ and $N \in \N$. Then the dyadic grid $\cD$ generated by $\algo{DyadicGrid}(\cGLGL_{N}, \{a,b\}, \alpha)$ is symmetric and monotonic.
\end{corollary}


\subsection{Gradedness of the dyadic grids}

If a symmetric grid $\cG$ and an initial dyadic grid $\cD_0$ are locally quasi-uniform (see \eqref{eq:grid}), one readily infers from Proposition~\ref{lem:dyadic_ver2} that, due to the local uniform $(A,B)$-equivalence of $\cG$ and the dyadic grid $\cD$, generated by the algorithm $\algo{DyadicGrid}(\cG, \cD_0, \alpha)$, the grid $\cD$ is also locally quasi-uniform. Numerical experiments indicate that this can be further quantified in terms the following notion of \emph{gradedness}.
\begin{definition}[Gradedness]
A dyadic grid $\cD$ of an interval $[a,b]$ is called \emph{graded}, if the levels of two neighboring dyadic intervals $D, D' \in \cP=\cP(\cD)$ differ at most by $1$, i.e., $1/2 \le \abs{D}/\abs{D'} \le 2$.
\end{definition}

\begin{lemma}
\label{lem:graded}
Assume that $\cG$ is a symmetric, monotonic, and locally quasi-uni\-form grid. If $\cD_0$ is graded and if the constant $C_g$ from \eqref{eq:grid} satisfies
\begin{equation}
\label{eq:alphaG}
\alpha C_g \leq 2,
\end{equation}
then the output $\cD$ of $\algo{DyadicGrid}(\cG, \cD_0, \alpha)$ is graded.
\end{lemma}
\begin{proof}
Consider any two neighboring dyadic intervals $D, D'\in \cP=\cP(\cD)$, $D, D'\subset [a,(a+b)/2]$, where as before $D$ is the left neighbor of $D'$. By Proposition~\ref{lem:MonotonicityDyadic}, we know that $\fsabs{D} \leq \fsabs{D'}$ and since the lengths of the dyadic intervals can only differ by powers of two, it remains to show that $\fsabs{D'} < 4 \fsabs{D}$. Suppose now to the contrary that
\begin{equation}
\label{eq:contra}
\fsabs{D'} \geq 4 \fsabs{D}.
\end{equation}
When $D\in \cP_0=\cP(\cD_0)$ there is nothing to prove since either both $D,D'\in \cP_0$, in which case \eqref{eq:contra} contradicts the hypothesis on $\cP_0$, or $D'\not\in\cP_0$ which means that the original right neighbor of $D$ in $\cP_0$ has been refined which also contradicts \eqref{eq:contra}. Now let $\Delta_l\in \cP$ denote the left neighbor of $\uDelta(D',\cG)$. From \eqref{eq:grid} and \eqref{eq:Ismall} we infer that
\begin{align*}
\fsabs{\Delta_l} \geq C_g^{-1} \fsabs{\uDelta(D',\cG)} \geq (C_g\alpha)^{-1} \fsabs{D'} \geq \frac 12 \fsabs{D'} \geq 2 \fsabs{D}.
\end{align*}
Thus, denoting by $\hat D$ the parent of $D$, we conclude that $\Delta_l = \uDelta(\hat D,\cG)$. Hence, since $\hat D\not\in \cP_0$, \eqref{eq:I'big} and the previous estimate assert that
\begin{align*}
2 \fsabs{D} = \fsabs{\hat D} > \alpha \fsabs{\Delta_l} \geq 2 \alpha \fsabs{D},
\end{align*}
which is a contradiction for $\alpha\geq 1$. This finishes the proof. \qed
\end{proof}

In order to apply this to LGL grids we note first that the above argument is local in the following sense. Again, considering by symmetry only the left half of the base interval, suppose that \eqref{eq:alphaG} holds only for intervals of $\cP(\cG)$ that are equal to or on the right of some $\tilde\Delta\in \cP(\cG)$. Then the above argument implies gradedness of $\cD$ for all $D,D'$ for which $\uDelta(\hat D,\cG)$ agrees with or is on the right of $\tilde \Delta$.

\begin{proposition}[Gradedness for LGL companion grids]
\label{prop:graded}
Let $1 \le \alpha {\le 1.25}$. Then the dyadic grids generated by the algorithm
$\algo{DyadicGrid}(\cGLGL_{N}, \cD_0, \alpha)$ are graded for {any $ 2\leq N \leq 2000$}, whenever $\cD_0$ is graded.

Moreover, when Conjecture~\ref{conj:Quotients} is valid, then the grids $\algo{DyadicGrid}(\cGLGL_{N},\allowbreak \cD_0,\alpha)$ are graded for all $N\in \N$.
\end{proposition}
\begin{proof}
Recall that property \MQ has been verified numerically to hold at least for $\bar N \leq 2000$. Thus, we can invoke Remark~\ref{rem:SmallestConstant} and note that for $\tilde C_g$ given by \eqref{eq:inner2} and \eqref{eq:q-2000} implies that for $\alpha\leq 1.25$ that condition \eqref{eq:alphaG} is satisfied for all quotients not involving the outermost LGL intervals. The same arguments apply, on account of Proposition~\ref{prop:Muldoon}, for all $N\in \N$ provided that
Conjecture~\ref{conj:Quotients}, holds.

Hence, it suffices to verify gradedness for the intervals adjacent to the left end point of the base interval. Let $D_0\in \cP=\cP(\cD)$ share the left end point. Then either $D_0$ has a right sibling of equal size in $\cP$ or equals half the base interval. In both cases gradedness holds trivially. The next observation is that the right neighbor $D_2\in \cP$ of the parent $\hat D$ of two siblings next to the left end point of the base interval must satisfy $\fsabs{D_2}\leq \fsabs{\hat D}$ since otherwise these intervals cannot belong to a partition that stems from successive dyadic splittings. The next possibility for breaking gradedness would be the transition to $D_3$. But, again to be part of the leaf set of a dyadic tree one must have $\fsabs{D_3}\leq \fsabs{D_0}+\fsabs{D_1}+\fsabs{D_2}=2\fsabs{D_2}$ which again implies gradedness of $\{D_0, D_1,D_2, D_3\}$. To obtain a jump of two levels at the left boundary of $D_3$ one must have that the two children $D_{2,0}, D_{2,1}$ also belong to $\cP$, i.e. $\fsabs{D_{2,i}}=\fsabs{D_1}=\fsabs{D_0}$, $i\in\{0,1\}$. But this means (since as in the proof of Lemma~\ref{lem:graded} we can assume that $D_{2,1}\not\in\cP_0$) that $\fsabs{D_2}> \alpha\fsabs{\uDelta(D_2,\cG)}$. Therefore, since $\alpha \geq 1$ and $\uDelta(D_2,\cG)\cap D_3\neq \emptyset$, we see that $\uDelta(D_2,\cG)$ does not contain the left end point of $[a,b]$ and hence is not an extreme interval. Since under the assumption that Conjecture~\ref{conj:Quotients} holds for all interior intervals of $\cP$ Lemma~\ref{lem:graded} applies, this finishes the proof. \qed
\end{proof}

Numerical evidence suggests that the constraint $\alpha \leq 1.25$, used above, is not necessary.


\subsection[Closedness of the dyadic grids under operator L]{Closedness of the dyadic grids under stretching}\label{sect:DyadicClosedL}

Let us recall the stretching operator $L=L_{[a,b]}:[a,(a+b)/2]\rightarrow[a,b], x \mapsto 2x-a$ as defined in Section~\ref{sect:B}. Under the assumption that Conjecture~\ref{conj:LGLComparisonStretched} is true, we can show that for $\alpha \ge 1$ and an LGL input grid $\cGLGL_N$, the dyadic grid generated by Algorithm~\ref{algo:Dyadic} is closed under stretching, in the sense of Definition~\ref{def:closedness}.

\begin{proposition}
\label{prop:stretched}
Let $\alpha \ge 1$, $N \in \N$ and $\cD=\algo{DyadicGrid}(\cGLGL_{N}, \cD_0, \alpha)$, where $\cD_0$ is closed under stretching. Then the validity of Conjecture~\ref{conj:LGLComparisonStretched} implies that $\cD$ is closed under stretching.
\end{proposition}
\begin{proof}
As usual, we set $\cP=\cP(\cD)$ and $\cP_0=\cP(\cD_0)$. By affine invariance it suffices to consider $[a,b]=[-1,1]$, see Remark~\ref{rem:DyadicAffineInvariance}. First note that since $L([-1,-1/2])= [-1,0]$ and $L([-1/2,0])= [0,1]$, it suffices, on account of symmetry and monotonicity, to show that $L(D)\in T(\cP)$ for any $D \in \cP$, $D\subseteq [-1,-1/2]$. For $D\in \cP_0\cap \cP$ there is nothing to show by our assumption on $\cD_0$. Suppose now that $D\in T(\cP)$ is a node that is split during the execution of algorithm $\algo{Dyadic}$, i.e., due to the condition in line~\ref{inalgo:split} of Algorithm~\ref{algo:Dyadic}, we know that $\fsabs{D} >\alpha \fsabs{\uDelta}$, where $\uDelta := \uDelta(D,\cG)$. Since $\alpha \geq 1$ we have, in particular, $\fsabs{\uDelta} \leq \fsabs{D}$. The assertion follows as soon as we have shown that the stretched version $L(D)$ must also be split, i.e., we have to show that
\begin{equation}
\label{eq:LDsplit}
\fsabs{D} >\alpha \fsabs{\uDelta(D,\cG)} \quad \Longrightarrow\quad \fsabs{L(D)} > \alpha \fsabs{\uDelta(L(D),\cG)}.
\end{equation}
To show this, let us observe first that it suffices to verify \eqref{eq:LDsplit} for $D\subseteq [-1,-1/2]$. In fact, any $D \subseteq [-1/2,0]$ gets mapped by $L$ into $[0,1]$. It then follows from the monotonicity and symmetry of $\cD$ that the midpoint of such a $D$ must be contained in $\cD \cap [0,1]$. So it remains to consider $D\subseteq [-1,-1/2]$. First, there is nothing to show when the left end point of $D$ is $-1$, since then $L(D)$ is the parent of $D$. We may therefore assume that $D$ does not contain $-1$. By the above comment $-1\not\in \uDelta(D,\cG)$. Clearly, $L(\uDelta(D,\cG))$ contains the left end point of $L(D)$ and therefore has to intersect $\uDelta(L(D),\cG)$. Under the assumption that Conjecture~\ref{conj:LGLComparisonStretched} is valid, it follows that
\begin{align*}
\fsabs{\uDelta(L(D),\cG)} \leq \fsabs{L(\uDelta(D,\cG))} = 2 \fsabs{\uDelta(D,\cG)} < 2 \alpha^{-1} \fsabs{D} = \alpha^{-1} \fsabs{L(D)},
\end{align*}
which finishes the proof. \qed
\end{proof}


\subsection{Construction of $\cD_N$ and the proof of Theorem~\ref{thm:B}}
A first natural attempt to construct a dyadic grid associated with a given LGL grid $\cGLGL_N$ would be to take $\cD_N = \algo{DyadicGrid}(\cGLGL_{N},\{a,b\},\alpha)$ for some $\alpha \in [1,1.25]$. In fact, the initial dyadic grid $\{a,b\}$ trivially satisfies all the assumptions on $\cD_0$ used in the derivation of the various properties above. Unfortunately, although this seems to occur very rarely, the grids $\algo{DyadicGrid}(\cGLGL_{N},\allowbreak\{a,b\},\allowbreak\alpha)$ are \emph{not} always nested, as shown by numerical evidence. For instance, for $\alpha=1$, the first pair $N^-, N^+$ of polynomial degrees where such dyadic grids are not nested, occurs for $N^+=20$ and $N^-=19$. For corresponding extensive numerical studies and further examples of non-nestedness, we refer the reader to \cite{brix-thesis}.

Therefore, to ensure nestedness we employ $\algo{DyadicGrid}(\cGLGL_{N}, \cD_0,\alpha)$ with dynamically varying initial grids $\cD_0$, as described in Algorithm~\ref{algo:DyadicNested}.

\begin{algorithm}[ht]
\begin{algorithmic}[1]
\State{$\cD_1 \gets \algo{DyadicGrid}(\cGLGL_{1}, \{a,b \}, \alpha)$}
\Comment{initialization}
\For{ $1 \le j< N$}
\State{$\cD_{j+1} \gets \algo{DyadicGrid}(\cGLGL_{j+1}, \cD_{j}, \alpha)$}
\Comment{refine $\cD_j$ for $\cGLGL_{j+1}$ according to \eqref{eq:Ismall}}
\EndFor
\end{algorithmic}
\caption{Algorithm $\algo{NestedDyadicGrid}(N,\{a,b\},\alpha)$ for the generation of LGL related nested dyadic grids.}
\label{algo:DyadicNested}
\end{algorithm}

By construction, the grids $\cD_N$ are nested. Moreover, one inductively concludes from the results of the preceding sections that the $\cD_N$ are symmetric and monotonic. They are also closed under stretching for any range of degrees $N$ for which property \SN holds. Moreover, they are also graded (beyond any numerically confirmed range of $N$) if Conjecture~\ref{conj:Quotients} is valid.

A little care must be taken to confirm the desired locally $(A,B)$-uniform equivalence of $\cD_N$ with $\cGLGL_N$. Again the lower inequality is ensured by \eqref{eq:Ismall}, see Remark~\ref{rem:relations} (i). As for the upper inequality, a certain obstruction lies in the fact that \eqref{eq:I'big} is not necessarily inherited for the specific value $\alpha$. In fact, the following situation may occur which again is a consequence of the fact that intervals in LGL grids not only decrease in size but also move outwards with increasing degree. Let for a given (dyadic) interval $D$, $\Delta = \uDelta(D,\cGLGL_{N})$ be the $\ell$-th interval in $\cP_N:=\cP(\cGLGL_{N})$. Then it could happen that $D$ no longer intersects the $\ell$-th interval in $\cP_{N+1}$, i.e. $\uDelta(D,\cGLGL_{N+1})$ is the $(\ell+1)$-st interval in $\cP_{N+1}$ and may therefore have larger size than $\uDelta(D,\cGLGL_{N})$. As a consequence $\fsabs{D} > \alpha \fsabs{\uDelta(D,\cGLGL_{N+1})}$ is not necessarily true. Nevertheless, the following can be shown.

\begin{property}
\label{lem:dyadic-2}
For all $N>1$, the dyadic grids $\cD_{N}$ produced by Algorithm~\ref{algo:DyadicNested} are locally $(A,B)$-uniformly equivalent to $\cGLGL_N$ with constants $A,B$ specified below:
\begin{multline}
\label{eq:dyeq2}
\qquad\forall D \in \cP_N=\cP(\cD_N) \;, \ \ \forall \,\Delta \in \cP_N=\cP(\cGLGL_N) \;,\\ \quad \Delta \cap D \neq \emptyset
\implies~
\alpha^{-1}\leq \frac{\fsabs{\Delta}}{\fsabs{D}} \leq \frac{2C_g}{\min\{\alpha C_g^{-1},1\}} \;.\qquad
\end{multline}
Furthermore, we have
\begin{equation}
\label{eq:dyeq2bis}
\#\, \cD_N \simeq \#\, \cGLGL_N \;.
\end{equation}
\end{property}
\begin{proof}
{Due to Remark~\ref{rem:DyadicAffineInvariance} we can assume without loss of generality that $[a,b]=[-1,1]$} and, by symmetry, consider only intervals in the left half $[-1,0]$. To be able to apply Proposition~\ref{lem:dyadic_ver2}, we shall exploit the fact that with increasing $N$ outward moving intervals decrease in size. More precisely, {by Theorem~\ref{theo:DispLGL2}}, one has for any $m\in \N$
\begin{equation}
\label{eq:left}
\begin{aligned}
\Delta^N_i:=[\xi^N_{i-1},\xi^N_i]\in \cP_N,\,\ \
\Delta^{N+m}_j= [\xi^{N+m}_{j-1},\xi^{N+m}_{j}]\in \cP_{N+m},\\
\xi^{N+m}_{j-1} \leq \xi^N_{i-1} \quad
\implies~
\fsabs{\Delta^{N+m}_j} \leq \fsabs{\Delta^N_i}\;.
\end{aligned}
\end{equation}
Now, on account of Proposition~\ref{lem:dyadic_ver2}, the assertion follows as soon as we have shown that \eqref{eq:I'bigg_ver2} holds with $\beta = \alpha C_g^{-1}$. To that end, suppose $D\in \cP_N$ is not subdivided in $\algo{DyadicGrid}(\cGLGL_{N+1}, \cD_N, \alpha)$. Without loss of generality we may assume that $\cP_N \neq \{[-1,1]\}$. Thus, $D$ must have been created by splitting $ \hat D(D)\in \cP_{N-m}$ for some $m\in \N$. By the condition in line~\ref{inalgo:split} of Algorithm~\ref{algo:Dyadic}, $\hat D(D)$ satisfies $\fsabs{\hat D(D)} > \alpha \fsabs{\uDelta(\hat D(D),\cGLGL_{N-m})}$. Now let $\Delta'\in \cP_{N-m}$ be the right neighbor of $\uDelta(\hat D(D),\cGLGL_{N-m})$, so that by \eqref{eq:grid}, $\fsabs{\hat D(D)} \geq \alpha C_g^{-1} \fsabs{\Delta'}$. One readily concludes from the monotonicity of the intervals in the LGL grids, see Theorem~\ref{theo:monotonicitylength}, that the left end point of $\uDelta(\hat D(D),\cGLGL_{N})$ must be smaller than the left end point of $\Delta'$. Therefore, we infer from \eqref{eq:left} that
\begin{align*}
\fsabs{\hat D(D)} \geq \alpha C_g^{-1} \fsabs{\Delta'} \geq C_g^{-1}\alpha \fsabs{\uDelta(\hat D(D),\cGLGL_{N})} \;,
\end{align*}
which is \eqref{eq:I'bigg_ver2}. This finishes the proof of \eqref{eq:dyeq2}. Finally, \eqref{eq:dyeq2bis} follows immediately from this result. \qed
\end{proof}

In summary, all the claims stated in Theorem~\ref{thm:B} have been verified to be satisfied by the dyadic grids $\cD_N$, generated by Algorithm~\ref{algo:DyadicNested}, which completes the proof of Theorem~\ref{thm:B}. \qed

\begin{figure}
\centering
\includegraphics[width=7.6cm]{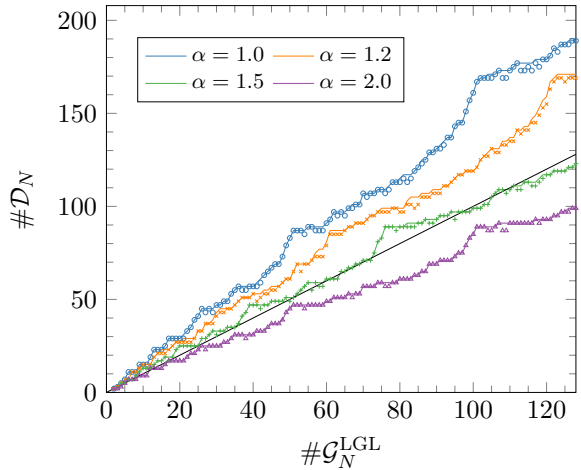}
\caption{Size of the nested dyadic grid $\cD_N$ (solid lines) as a function of the size of $\cGLGL_{N}$ for different values of $\alpha$. The sizes of the dyadic grids as in Figure~\ref{fig:GridsizesDyadic} are also given for a comparison (marks). The black graph indicates the line of equal cardinalities of LGL and dyadic meshes.}
\label{fig:GridsizesDyadicNested}
\end{figure}

\smallskip
In Figure~\ref{fig:GridsizesDyadicNested} the sizes of the nested dyadic grids $\cD_N$ are plotted against the sizes of the LGL grids $\cGLGL_{N}$ for different values of $\alpha$. We observe that usually only very few points, in comparison with running just $\algo{Dyadic}(\cGLGL_N, \{a,b\}, \alpha)$, are added to the grid to ensure nestedness.

Finally, we can invoke Proposition~\ref{prop:graded} to confirm the claims in Remark~\ref{rem:graded} for the family of dyadic grids $\algo{DyadicGrid}(\cGLGL_{N}, \{a,b \}, \alpha)$, $1\leq \alpha \leq 1.25$.


\begin{acknowledgements}
This work was supported in part by the DFG project ``Optimal preconditioners of spectral Discontinuous Galerkin methods for elliptic boundary value problems'' (DA 117/23-1), the Excellence Initiative of the German federal and state governments (RWTH Aachen Seed Funds, Distinguished Professorship projects, Graduate School AICES), and NSF grant DMS 1222390.

We are indebted to Martin E. Muldoon for inspiring discussions concerning the spacing of LGL nodes and to Sabrina Pfeiffer for helpful comments on the manuscript.
\end{acknowledgements}


\end{document}